%% file: Takase-20200829.tex
\documentclass[a4paper,10pt]{article}

\usepackage{amsmath}  
\usepackage{amssymb}            
\usepackage{amsxtra}
\usepackage{amscd}
\usepackage[mathscr]{euscript}
\usepackage{cite}
\usepackage[dvipdfmx]{graphicx}
\usepackage{tikz}
\usepackage{tikz-cd}
\usepackage[dvipdfmx,%
            pagebackref,
            ]{hyperref}

\everymath{\displaystyle} 

\newtheorem{thm}{Theorem}[subsection]
\newtheorem{prop}[thm]{Proposition}
\newtheorem{cor}[thm]{Corollary}

\newtheorem{rem}[thm]{Remark}

\numberwithin{equation}{subsection}

\newcommand{\npmod}[1]{\!\!\pmod{#1}}
\newcommand{\nnpmod}[1]{\!\!\!\!\pmod{#1}}

\newenvironment{proof}{\par\noindent{\bf[Proof]}}%
                      {$\blacksquare$\noindent\par\vspace{0.5\baselineskip}}
                      {$\blacksquare$\par\noindent}

\font\b=cmr10 scaled \magstep4

\def\bigzerou{\smash{\lower0.7ex\hbox{\b 0}}}
\def\bigastl{\smash{\lower0.7ex\hbox{\b *}}}
\def\bigastu{\smash{\lower2.7ex\hbox{\b *}}}

\def\addots{\mathinner
    {\mkern1mu\raise1pt\hbox{.}\mkern2mu
    \raise4pt\hbox{.}\mkern2mu\raise7pt\vbox{\kern7pt\hbox{.}}\mkern1mu}}

\title{Regular irreducible representations \\
        of classical groups\\
       over finite quotient rings}
\author{Koichi Takase
            \footnote{The author is partially supported by 
                      JSPS KAKENHI Grant Number JP 16K05053,\newline
                      }}
\date{}

\begin{document}   


%
%

\maketitle


\begin{abstract}
\input{abstract.tex}
\end{abstract}

\input{intro.tex}
\input{results.tex}

\input{regularity.tex}
\input{schur.tex}

\input{weilrep.tex}
\input{classical.tex}

\input{ref.tex}


\begin{flushleft}
Sendai 980-0845, Japan\\
Miyagi University of Education, 
Department of Mathematics\\
e-mail : k-taka2@ipc.miyakyo-u.ac.jp
\end{flushleft}         
\end{document}

%% file: abstract.tex
A parametrization of irreducible representations 
associated with a regular adjoint orbit of
a classical group over finite quotient rings of the ring of integer of
a non-dyadic
non-archimedean local field 
is presented. The parametrization is given by means of (a subset of) 
the character group of the centralizer of a representative of the 
regular adjoint orbit. 
Our method 
is based upon Weil representations over finite fields. 
More explicit parametrization in terms of tamely ramified extensions of
the base field is given for the general linear group, the special
linear group, the symplectic group and the orthogonal group.

%% file: intro.tex
\section{Introduction}\label{sec:introduction}

Let $F$ be a non-dyadic non-archimedean local field. 
The integer ring of $F$ is denoted by $O=O_F$ with the maximal
ideal $\frak{p}=\frak{p}_F$ generated by $\varpi$. The residue class field 
$\Bbb F=O/\frak{p}$ is a finite field of odd characteristic with 
$q$ elements.
For an integer $r>0$ put $O_r=O/\frak{p}^r$ so that $\Bbb F=O_1$. 

Let $G$ be a connected reductive group scheme over $O$. The problem
which we will consider in this paper is to determine the set 
$\text{\rm Irr}(G(O_r))$ of the equivalence classes of the 
irreducible complex representations of the finite group $G(O_r)$. 

This problem in the case $r=1$, that is the representation theory of
the finite reductive group $G(\Bbb F)$, has been studied extensively,
starting 
from Green \cite{Green1955} concerned with $GL_n(\Bbb F)$ to the
decisive paper of Deligne-Lusztig \cite{Deligne-Lusztig1976}. 

On the other hand, the study of the representation
theory of the finite group $G(O_r)$ with $r>1$ is less complete. 
The systematic studies are done mainly in the case of $G=GL_n$ 
\cite{Hill1993,Hill1994,Hill1995-1,Hill1995-2}, 
\cite{Krakovski2018}, \cite{Stasinski-Stevens2017}, 
\cite{Takase2016}. 

Shechter \cite{Shechter2019} studies the problem when $G$ is
a symplectic group or a special orthogonal group with $r>1$.
We have a canonical isomorphism
\begin{equation}
  K_{r-1}(O_r)\,\tilde{\to}\,\frak{g}(\Bbb F)
\label{eq:canonical-isom-of-k(r-1)-to-g(f)}
\end{equation}
where $K_{r-1}(O_r)$ is the kernel of the canonical group homomorphism
$G(O_r)\to G(O_{r-1})$ and $\frak{g}$ is the Lie algebra scheme of
$G$. Take an irreducible representation $\pi$ of $G(O_r)$. Then, by
Clifford's theorem, the restriction $\pi|_{K_{r-1}(O_r)}$ is a
multiple of the sum over a single $G(\Bbb F)$-conjugacy class of characters of 
$K_{r-1}(O_r)$ which determines, by the isomorphism
\eqref{eq:canonical-isom-of-k(r-1)-to-g(f)} combined with the
invariant bilinear form on $\frak{g}(\Bbb F)$, an adjoint 
$G(\Bbb F)$-orbit $\Omega$ in $\frak{g}(\Bbb F)$. So we have a
correspondence $\pi\mapsto\Omega$ of the set of the equivalence
classes of the irreducible representations of $G(O_r)$ to the set of
the adjoint $G(\Bbb F)$-orbits in $\frak{g}(\Bbb F)$. 
Shechter \cite{Shechter2019} constructs (when $G$ is a symplectic
group or a special orthogonal group) the irreducible representations
of $G(O_r)$ ($r>1$) which correspond to an adjoint $G(\Bbb F)$-orbit
$\Omega$ consisting of regular elements of $\frak{g}(\Bbb F)$ 
(see section \ref{sec:regularity-of-lie-element} for the definition of
  a regular Lie element). 

In this paper, we will treat more generally any smooth $O$-group
scheme $G$ with Lie algebra scheme $\frak{g}$ 
which satisfies three fundamental conditions I), II) and III)
presented in subsection \ref{subsec:fundamental-setting}. We will
also assume that $G$ is an $O$-group subscheme of $GL_n$
suitably and hence $\frak{g}$ is a closed $O$-subscheme of 
$\frak{gl}_n$ the Lie algebra scheme of $GL_n$. 
The condition
II) gives explicitly the isomorphism 
\begin{equation}
 K_l(O_r)\,\tilde{\to}\,\frak{g}(O_{l^{\prime}})
\label{eq:canonical-isom-of-k(l)-to-g(l-prime)}
\end{equation}
where $l^{\prime}=l$ if $r=2l$ is even and $l^{\prime}=l-1$ if
$r=2l-1>1$ is odd, and $K_l(O_r)$ is the kernel of the canonical group
homomorphism $G(O_r)\to G(O_l)$ which is surjective due to the formal
smoothness \cite[p.111, Cor. 4.6]{Demazure-Gabriel1970}. The condition
I) guarantees the existence of an invariant bilinear form on 
$\frak{g}(O_{l^{\prime}})$. 
Then the restriction of an irreducible representation
of $G(O_r)$ to $K_l(O_r)$ determines an adjoint 
$G(O_{l^{\prime}})$-orbit in $\frak{g}(O_{l^{\prime}})$ by Clifford's
theorem as above. Let 
$\Omega^{\prime}\subset\frak{g}(O_{l^{\prime}})$ be the adjoint 
$G(O_{l^{\prime}})$-orbit of 
$\beta\npmod{\frak{p}^{l^{\prime}}}\in\frak{g}(O_{l^{\prime}})$ with 
$\beta\in\frak{g}(O)$. Our main result establishes a bijection between
  the set of the equivalence classes of the irreducible
  representations of $G(O_r)$ corresponding to $\Omega^{\prime}$ and
  a subset of the character group of $G_{\beta}(O_r)$, where 
$G_{\beta}$ is the centralizer of $\beta$ in $G$, under the assumptions 
\begin{enumerate}
\item $G_{\beta}$ is smooth over $O$, and 
\item the characteristic polynomial of 
      $\beta\npmod{\frak p}\in M_n(\Bbb F)$ is its minimal polynomial 
      where we put 
      $\frak{g}(\Bbb F)\subset\frak{gl}_n(\Bbb F)=M_n(\Bbb F)$.
\end{enumerate}
Due to the second condition, the centralizer $GL_{n,\beta}$ of 
$\beta\in\frak{gl}_n(O)$ in $GL_n$ is commutative (see subsection 
\ref{subsec:detailed-description-of-the-case-gl-n}), and hence 
$G_{\beta}$ is a commutative $O$-group scheme. See section 
\ref{sec:regularity-of-lie-element} for relations between the
second condition and the regularity (or the smooth regularity, 
according to Springer \cite{Springer1966}) of Lie elements. 

Let us consider the inverse image of an adjoint $G(\Bbb F)$-orbit 
$\Omega$ in $\frak{g}(\Bbb F)$ under the canonical surjection 
$\frak{g}(O_{l^{\prime}})\to\frak{g}(\Bbb F)$. The inverse image
decomposes into several adjoint $G(O_{l^{\prime}})$-orbits consisting
of the same number of elements. Then the set of the equivalence
classes of irreducible representations of $G(O_r)$ corresponding to
$\Omega$ are divided into several subclasses corresponding to these 
adjoint $G(O_{l^{\prime}})$-orbits. 
In other words, 
our main result, applied to the case where $G$ is a symplectic group
or a special orthogonal group, divides Shechter's irreducible
representations into several subclasses and gives a parametrization of
the irreducible representations in each subclasses. 
In particular we will construct the irreducible representations of 
Shechter \cite{Shechter2019} 
(see subsection \ref{subsec:discussionon-shechter-results}).

The main results of this paper are Theorem \ref{th:main-result}
 and its applications to the
general linear group $GL_n$ (Theorem \ref{th:main-result-for-gl-n}, 
Theorem
\ref{th:main-result-for-gl-n-related-with-tamely-ramified-extension}), 
to the special linear group $SL_n$ (Theorem
\ref{th:shintani-gerardin-parametrization-for-sl-n}), to the
symplectic group $Sp_{2n}$ (Theorem
\ref{thm:shintani-gerardin-type-parametrization-for-sp-2n}) and to the
orthogonal groups $SO_{2n}$ (Theorem
\ref{thm:shintani-gerardin-type-parametrization-for-even-so}) and 
$SO_{2n+1}$ (Theorem
\ref{thm:shintani-gerardin-type-parametrization-for-odd-so}). 

The situation is quite simple and well known when $r$ is even, and
almost all of this 
paper is devoted to study the case of $r=2l-1$ being odd. In this case
Clifford theory requires us to construct an irreducible representation 
of $G_{\beta}(O_r)\cdot K_{l-1}(O_r)$. 
To construct an irreducible representation of
$K_{l-1}(O_r)$, 
we will use Schr\"odinger representation of the Heisenberg group
associated with a symplectic space over finite field which is
associated with $\beta$ 
(Proposition \ref{prop:another-expression-of-pi-beta-psi}). 
Then we will use Weil representation to
extend the irreducible representation of $K_{l-1}(O_r)$ to a 
projective representation of $G_{\beta}(O_r)\cdot K_{l-1}(O_r)$. 
Finally we will prove that the Schur multiplier associated with the
projective representation is trivial 
(Proposition \ref{prop:vanishing-of-c-beta-rho}) to get the required
irreducible linear representation of $G_{\beta}(O_r)\cdot K_{l-1}(O_r)$. 

In the case of $G=GL_n$, the extendability of the
irreducible representation of $K_{l-1}(O_r)$ to that of 
$G_{\beta}(O_r)\cdot K_{l-1}(O_r)$ is proved by 
\cite{Krakovski2018}, \cite{Stasinski-Stevens2017}. 
Based upon this result, we will prove
the triviality of the Schur multiplier for general $G\subset GL_n$
under the condition that the reduction modulo $\frak{p}$ of the
characteristic polynomial of 
$\beta\in\frak{g}(O)\subset\frak{gl}_n(O)$ is the minimal polynomial
of $\beta\npmod{\frak p}\in M_n(\Bbb F)$ (in this case 
$\beta\npmod{\frak p}\in\frak{gl}_n(\Bbb F)$ is regular with
respect to $GL_n$ over $\Bbb F$, 
see subsection \ref{subsec:detailed-description-of-the-case-gl-n}). 

Note that Weil representation over a finite field is a ``genuine'' linear
representation, and the definition of the Schur multiplier is
independent of the theory of Weil representation. 
The definition and 
fundamental properties of the Schur multiplier will be
discussed in section \ref{sec:schur-multiplier}.


{\bf Notations}
The multiplicative group of the complex
numbers of absolute value one is denoted by $\Bbb C^1$. 
The character group of a finite abelian group $\mathcal{G}$,
 that is the multiplicative group consisting of the group
 homomorphisms of $\mathcal{G}$ to $\Bbb C^1$, 
is denoted by $\mathcal{G}\sphat$. 

For any $O$-scheme $X$, the image of $x\in X(O)$ under the canonical
mapping $X(O)\to X(O_l)$ is denoted by $x_l=x\pmod{\frak{p}^l}$. On the
other hand, when $K=\Bbb F$ or $F$, the image of $x\in X(O)$ under the
canonical mapping $X(O)\to X(K)$ is denoted by $\overline x$. Any
notational conflict between the cases $K=\Bbb F$ and $K=F$ may not
occur in this paper.

{\bf Acknowledgment} The author express his thanks to the referee who
read quite carefully the submitted manuscript and gave suggestions to
improve greatly this article. Particularly a critical step of the
proof of Proposition \ref{prop:vanishing-of-c-beta-rho} is suggested
by the referee. 

%% file: results.tex
\section{Main results}\label{sec:main-results}

\subsection[]{Fundamental assumptions}\label{subsec:fundamental-setting}
Let $G\subset GL_n$ be a closed smooth $O$-group subscheme, and
$\frak{g}$ the Lie algebra scheme of $G$ which is a closed affine
$O$-subscheme of $\frak{gl}_n$ the Lie algebra scheme of $GL_n$. We 
  assume that the fibers $G{\otimes}_OK$ ($K=F$ or $K=\Bbb F$) are
  non-commutative algebraic $K$-groups (that is smooth $K$-group
  schemes).  

For any $O$-algebra $R$ (in this paper, an $O$-algebra means an
commutative unital $O$-algebra) 
the set of the $R$-valued points $\frak{gl}_n(R)$ is
  identified with the $R$-Lie algebra of square matrices $M_n(R)$ of size $n$
  with Lie bracket $[X,Y]=XY-YX$, and 
the group of $R$-valued points $GL_n(R)$ is
identified with the matrix group
$$
 GL_n(R)=\{g\in M_n(R)\mid \det g\in R^{\times}\}
$$  
where $R^{\times}$ is the multiplicative group of $R$. Hence
$\frak{g}(R)$ is identified with a matrix Lie subalgebra of $\frak{gl}_n(R)$
and $G(R)$ is identified with a matrix subgroup of $GL_n(R)$. Let 
$$
 B:\frak{gl}_n{\times}_O\frak{gl}_n\to\Bbb A_O^1
$$
be the trace form on $\frak{gl}_n$, that is $B(X,Y)=\text{\rm tr}(XY)$
for all $X,Y\in\frak{gl}_n(R)$ with any $O$-algebra $R$. The
smoothness of $G$ implies that we have a canonical isomorphism
$$
 \frak{g}(O)/\varpi^r\frak{g}(O)\,\tilde{\to}\,
 \frak{g}(O_r)=\frak{g}(O){\otimes}_{O}O_r
$$
(\cite[Chap.II, $\S 4$, Prop.4.8]{Demazure-Gabriel1970}) and that the
 canonical group homomorphism $G(O)\to G(O_r)$ is
 surjective due to the formal smoothness 
\cite[p.111, Cor. 4.6]{Demazure-Gabriel1970} and the canonical
isomorphism 
$O\,\tilde{\to}\,\lim_{\stackrel{\longleftarrow}{r}}O_r$. 
For any $0<l<r$, let us denote by $K_l(O_r)$ the kernel of the canonical
 group homomorphism $G(O_r)\to G(O_l)$ which is surjective.

Throughout this paper, 
we will assume the following three conditions;
\begin{itemize}
\item[I)] $B:\frak{g}(\Bbb F)\times\frak{g}(\Bbb F)\to\Bbb F$ is
  non-degenerate, 
\item[II)] for any integers $r=l+l^{\prime}$ with 
           $0<l^{\prime}\leq l$, we have a group isomorphism
$$
 \frak{g}(O_{l^{\prime}})\,\tilde{\to}\,K_l(O_r)
$$
           defined by 
$X\npmod{\frak{p}^{l^{\prime}}}\mapsto1+\varpi^lX\npmod{\frak{p}^r}$,
\item[III)] if $r=2l-1\geq 3$ is odd, then we have a mapping
$$
 \frak{g}(O)\to K_{l-1}(O_r)
$$
defined by 
$X\mapsto(1+\varpi^{l-1}X+2^{-1}\varpi^{2l-2}X^2)\npmod{\frak{p}^r}$.
\end{itemize}
The condition I) implies that 
$B:\frak{g}(O_l)\times\frak{g}(O_l)\to O_l$ 
is non-degenerate for all $l>0$, and so 
$B:\frak{g}(O)\times\frak{g}(O)\to O$ is also non-degenerate. 
The mappings of the conditions II) and III) from Lie algebras to
groups can be regarded as truncations of the exponential mapping. 
Note that the mapping of the condition III) is not a group
homomorphism. This mapping plays an important role when we will 
analyze in subsection \ref{subsec:k-l-1-as-group-extension} 
the structure of $K_{l-1}(O_r)$ where $r=2l-1$ is odd. We will
use it also in the proof of Proposition
\ref{prop:fundamental-bijection-of-parameter}.

\subsection[]{Clifford theory}\label{subsec:clifford-theory-in-general}
Fix a continuous unitary character $\tau$ of the additive group $F$ such
that 
$$
 \{x\in F\mid\tau(xO)=1\}=O.
$$
We will fix an integer $r\geq 2$ and put 
$r=l+l^{\prime}$ with the smallest integer $l$ such that 
$0<l^{\prime}\leq l$. In other words
$$
 l^{\prime}=\begin{cases}
             l&:r=2l,\\
             l-1&:r=2l-1.
            \end{cases}
$$
Take a $\beta\in\frak{g}(O)$ and define a
character $\psi_{\beta}$ of the finite abelian group 
$K_l(O_r)$ by 
$$
 \psi_{\beta}((1+\varpi^lX)\nnpmod{\frak{p}^r})
 =\tau(\varpi^{-l^{\prime}}B(X,\beta))
 \qquad
 (X\in\frak{g}(O)).
$$
Then $\beta\npmod{\frak{p}^{l^{\prime}}}\mapsto\psi_{\beta}$ 
gives an isomorphism of the
additive group $\frak{g}(O_{l^{\prime}})$ onto the character
group $K_l(O_r)\sphat$. For any 
$g_r=g\npmod{\frak{p}^r}\in G(O_r)$ with $g\in G(O)$, we have
\begin{equation}
 \psi_{\beta}(g_r^{-1}hg_r)
 =\psi_{\text{\rm Ad}(g)\beta}(h)
 \qquad
 (h\in K_l(O_r)).
\label{eq:conjugate-of-psi-beta}
\end{equation}
So the stabilizer of $\psi_{\beta}$ in $G(O_r)$ is 
$$
 G(O_r,\beta)
 =\left\{g_r\in G(O_r)\bigm|
         \text{\rm Ad}(g)\beta\equiv\beta
                    \pmod{\frak{p}^{l^{\prime}}}\right\}
$$
which is a subgroup of $G(O_r)$ containing $K_{l^{\prime}}(O_r)$.

Now let us denote by 
$\text{\rm Irr}(G(O_r)\mid\psi_{\beta})$ 
(resp. $\text{\rm Irr}(G(O_r,\beta)\mid\psi_{\beta})$) the set of the
isomorphism classes of the irreducible complex representations $\pi$ of 
$G(O_r)$ (resp. $\sigma$ of $G(O_r,\beta)$) such
that 
$$
 \langle\psi_{\beta},\pi\rangle_{K_l(O_r)}
 =\dim_{\Bbb C}\text{\rm Hom}_{K_l(O_r)}(\psi_{\beta},\pi)
 >0
$$
(resp. 
$\langle\psi_{\beta},\sigma\rangle_{K_l(O_r)}>0$). Then we have
\begin{enumerate}
\item $\text{\rm Irr}(G(O_r))
       =\bigsqcup_{\beta\npmod{\frak{p}^{l^{\prime}}}}
           \text{\rm Irr}(G(O_r)\mid\psi_{\beta})$
  where $\bigsqcup_{\beta\npmod{\frak{p}^{l^{\prime}}}}$ 
  is the disjoint union over the
  representatives $\beta\npmod{\frak{p}^{l^{\prime}}}$ 
  of the $\text{\rm Ad}(G(O_{l^{\prime}}))$-orbits in 
  $\frak{g}(O_{l^{\prime}})$,
\item a bijection of 
      $\text{\rm Irr}(G(O_r,\beta)\mid\psi_{\beta})$ onto 
      $\text{\rm Irr}(G(O_r)\mid\psi_{\beta})$ is given by 
$$
 \sigma\mapsto\text{\rm Ind}_{G(O_r,\beta)}^{G(O_r)}\sigma.
$$
\end{enumerate}
The first statement is clear from the fact that 
$K_l(O_r)\,\tilde{\to}\,\frak{g}(O_{l^{\prime}})$ is an commutative
group and the relation \eqref{eq:conjugate-of-psi-beta}. For the
second statement, see \cite[Th.6.11]{Isaacs1976}.

So our problem is reduced to give a good parametrization of the 
   set 
$$
 \text{\rm Irr}(G(O_r,\beta)\mid\psi_{\beta}).
$$

\subsection[]{Main theorem}\label{subsec:first-main-result}
For any $\beta\in\frak{g}(O)$, let us denote
by $G_{\beta}=Z_G(\beta)$ the centralizer of
$\beta$ in $G$  which is a closed $O$-group subscheme of $G$. The Lie
algebra $\frak{g}_{\beta}=Z_{\frak g}(\beta)$ of $G_{\beta}$ is 
a closed $O$-subscheme of $\frak{g}$ such that
$$
 \frak{g}_{\beta}(R)
 =\{X\in\frak{g}(R)\mid [X,\overline\beta]=0\}
$$
for any $O$-algebra $R$ where $\overline\beta\in\frak{g}(R)$ is the
image of $\beta\in\frak{g}(O)$ under the canonical morphism 
$\frak{g}(O)\to\frak{g}(R)$. 

Now our main result is

\begin{thm}\label{th:main-result}
Take a $\beta\in\frak{g}(O)$ such that
\begin{enumerate}
\item $G_{\beta}$ is a smooth $O$-group scheme, and
\item the characteristic polynomial 
      $\chi_{\overline\beta}(t)=\det(t\cdot 1_n-\overline\beta)$ of 
      $\overline\beta\in\frak{g}(\Bbb F)\subset\frak{gl}_n(\Bbb F)$ is
  the minimal polynomial of $\overline\beta\in M_n(\Bbb F)$.
\end{enumerate}
Then 
we have a bijection $\theta\mapsto\sigma_{\beta,\theta}$ of the set
$$
 \left\{\theta\in G_{\beta}(O_r)\sphat\;\;\;
         \text{\rm s.t. $\theta=\psi_{\beta}$ on 
               $G_{\beta}(O_r)\cap K_l(O_r)$}\right\}
$$
onto $\text{\rm Irr}(G(O_r,\beta)\mid\psi_{\beta})$.
\end{thm}

\begin{cor}\label{cor:main-result}
Under the conditions of Theorem \ref{th:main-result}, we have a bijection 
$$
 \theta\mapsto\delta_{\beta,\theta}
 =\text{\rm Ind}_{G(O_r,\beta)}^{G(O_r)}\sigma_{\beta,\theta}
$$
of the set
$$
 \left\{\theta\in G_{\beta}(O_r)\sphat\;\;\;
         \text{\rm s.t. $\theta=\psi_{\beta}$ on 
               $G_{\beta}(O_r)\cap K_l(O_r)$}\right\}
$$
onto $\text{\rm Irr}(G(O_r)\mid\psi_{\beta})$.
\end{cor}

Note that $G_{\beta}$ in Theorem \ref{th:main-result} is commutative
because the centralizer of $\beta$ in $GL_n$ is commutative due to 
the second condition of the theorem. (see subsection
\ref{subsec:detailed-description-of-the-case-gl-n}). 

The explicit description of the representation $\sigma_{\beta,\theta}$
is given by 
\eqref{eq:description-of-sigma-beta-theta-for-even-r} if $r$ is even,
and by 
\eqref{eq:description-of-sigma-beta-theta-for-odd-r} if $r$ is odd. 

The proof of this theorem in the case of even $r$ is quite simple and
well known. For the sake of completeness, we will include a proof 
in the next subsection. The proof for the case
of odd $r$ is given in section \ref{sec:weil-representation} after
studying certain Schur multiplier in section
\ref{sec:schur-multiplier}. 


\subsection[]{Proof of Theorem \ref{th:main-result} for even $r$}
\label{subsec:proof-of-main-result-for-even-r}
Assume that $r=2l$ is even so that $l^{\prime}=l$. Since $G_{\beta}$
is a smooth $O$-group scheme, the canonical map 
$G_{\beta}(O_r)\to G_{\beta}(O_l)$ is surjective. Then we have 
$$
 G(O_r,\beta)=G_{\beta}(O_r)\cdot K_l(O_r)
$$
where $G_{\beta}(O_r)$ is a commutative finite group. 
Let $\theta$ be a character of $G_{\beta}(O_r)$ such that 
$\theta=\psi_{\beta}$ on $G_{\beta}(O_r)\cap K_l(O_r)$. Then we have 
an one-dimensional
representation $\sigma_{\beta,\theta}$ of $G(O_r,\beta)$ defined by 
\begin{equation}
 \sigma_{\beta,\theta}(gh)=\theta(g)\cdot\psi_{\beta}(h)
 \qquad
 (g\in G_{\beta}(O_r), h\in K_l(O_r)).
\label{eq:description-of-sigma-beta-theta-for-even-r}
\end{equation}
Then $\theta\mapsto\sigma_{\beta,\theta}$ is an injection of the set
$$
 \left\{\theta\in G_{\beta}(O_r)\sphat\;\;\;
         \text{\rm s.t. $\theta=\psi_{\beta}$ on 
               $G_{\beta}(O_r)\cap K_l(O_r)$}\right\}
$$
into $\text{\rm Irr}(G(O_r,\beta)\mid\psi_{\beta})$.

Take any $\sigma\in\text{\rm Irr}(G(O_r,\beta)\mid\psi_{\beta})$ with
representation space $V_{\sigma}$. Then
$$
 V_{\sigma}(\psi_{\beta})
 =\{v\in V_{\sigma}\mid \sigma(g)v=\psi_{\beta}(g)v\,
                \text{\rm for $\forall g\in K_l(O_r)$}\}
$$
is a non-trivial $G(O_r,\beta)$-subspace of $V_{\sigma}$ so that 
$V_{\sigma}=V_{\sigma}(\psi_{\beta})$. Then, for any 
one-dimensional representation $\chi$ of $G(O_r,\beta)$ such that 
$\chi=\psi_{\beta}$ on $K_l(O_r)$, we have 
$K_l(O_r)\subset\text{\rm Ker}(\chi^{-1}\otimes\sigma)$. On the other
hand $G(O_r,\beta)/K_l(O_r)$ is commutative, we have 
$\dim(\chi^{-1}\otimes\sigma)=1$ and then $\dim\sigma=1$. 
Then $\theta=\sigma|_{G_{\beta}(O_r)}$ is a character of
$G_{\beta}(O_r)$ such that $\theta=\psi_{\beta}$ on 
$G_{\beta}(O_r)\cap K_l(O_r)$, and 
we have $\sigma=\sigma_{\beta,\theta}$.

%% file: regularity.tex
\section{Regularity of Lie elements}
\label{sec:regularity-of-lie-element}
In this section, we will discuss the relation between the regularity
(or smooth regularity, according to 
\cite[p.138]{Springer1966}) of a Lie element and the two conditions of
Theorem \ref{th:main-result}.

\subsection[]{Regularity and smooth regularity}
\label{subsec:suff-condition-for-smooth-commutativeness-of-g-beta}

Let us assume that the connected $O$-group scheme $G$ is reductive,
that is, the fibers $G{\otimes}_OK$ ($K=F,\Bbb F$) are reductive
$K$-algebraic groups. In this case the dimension of a maximal torus
in $G{\otimes}_OK$ is independent of $K$ which is denoted by 
$\text{\rm rank}(G)$. For any $\beta\in\frak{g}(O)$ we have
\begin{equation}
 \dim_K\frak{g}_{\beta}(K)=\dim\frak{g}_{\beta}{\otimes}_OK
 \geq\dim G_{\beta}{\otimes}_OK\geq\text{\rm rank}(G).
\label{eq:dimension-of-centrlizer-on-lie-alg-and-group}
\end{equation}
We say $\beta$ is {\it smoothly regular} (resp. {\it regular}) 
with respect to $G$ over $K$ (or simply with respect to $G{\otimes}_OK$) 
if $\dim_K\frak{g}_{\beta}(K)=\text{\rm rank}(G)$ 
(resp. $\dim G_{\beta}{\otimes}_OK=\text{\rm rank}(G)$)
(see \cite[(5.7)]{Springer1966}). 
The relation \eqref{eq:dimension-of-centrlizer-on-lie-alg-and-group}
shows that if $\beta$ is smoothly regular with respect to 
$G{\otimes}_OK$ then $\beta$ is regular with respect to 
$G{\otimes}_OK$. If $\beta$ is smoothly regular with respect to 
$G{\otimes}_OK$ then $G_{\beta}{\otimes}_OK$ is smooth over $K$. 

If $\beta$ is smoothly regular with respect to $G$
over $F$ and over $\Bbb F$, then we say $\beta$ is  smoothly regular
with respect to $G$. 

We say $\beta$ is {\it connected} with respect to $G$ if the 
fibers $G_{\beta}{\otimes}_OK$ ($K=F, \Bbb F$) are connected. 
See Remark
\ref{remark:smooth-regularity-and-connectedness-of-centralizer} for a
sufficient condition for the connectedness of
$G_{\beta}{\otimes}_OK$. Then we have

\begin{prop}
\label{tprop:sufficient-condition-for-smooth-commutativeness-of-g-beta}
If $\beta\in\frak{g}(O)$ is smoothly regular and connected with
respect to $G$, then $G_{\beta}$ is smooth over $O$.
\end{prop}
\begin{proof}
Let $G_{\beta}^o$ be the neutral component of $O$-group scheme
$G_{\beta}$ which is a group functor of the category of $O$-scheme 
(see $\S 3$ of Expos\'e $\text{\rm VI}_B$ in \cite{SGA-3}).
The following statements are equivalent;
\begin{enumerate}
\item $G_{\beta}^o$ is representable as an smooth open $O$-group subscheme of
      $G_{\beta}$, 
\item $G_{\beta}$ is smooth at the points of unit section,
\item each fibers $G_{\beta}{\otimes}_OK$ ($K=F,\Bbb F$) are smooth
      over $K$ and their dimensions are constant
\end{enumerate}
(see Th. 3.10 and Cor. 4.4 of \cite{SGA-3}). So if $\beta$ is
smoothly regular with respect to $G$, then $G_{\beta}^o$ is smooth
open $O$-group subscheme of $G_{\beta}$. If further $\beta$ is connected
with respect to $G$, then $G_{\beta}^o=G_{\beta}$ is smooth over
$O$. 
\end{proof}

Let us present more detail description of the smooth regularity of Lie
element. Put $K=F$ or $\Bbb F$.

Take a $\beta\in\frak{g}(O)$ and let $\overline\beta=\beta_s+\beta_n$
be the Jordan decomposition of $\overline\beta\in\frak{g}( K)$
into the semi-simple part $\beta_s\in\frak{g}( K)$ and the
nilpotent part $\beta_n\in\frak{g}( K)$
($\overline\beta\in\frak{g}(K)$ is the image of $\beta\in\frak{g}(O)$
under the canonical mapping $\frak{g}(O)\to\frak{g}(K)$). 
The identity component 
$L=Z_{G{\otimes}_O K}(\beta_s)^o$ of the centralizer of $\beta_s$ in 
$G{\otimes}_O K$ is a reductive group over $ K$ and there
exists a maximal torus $T$ of $G{\otimes}_O K$ such that
$$
 \beta_s\in\text{\rm Lie}(T)( K)
$$
(see \cite[Prop.13.19]{Borel1991} and its proof). 
Then $T\subset L$ and 
$\text{\rm rank}(L)=\text{\rm rank}(G)$. Put 
$\frak{l}=\text{\rm Lie}(L)$, then 
$\frak{l}(K)=Z_{\frak{g}(K)}(\beta_s)$. 
So $\overline\beta\in\frak{g}( K)$ is smoothly regular with
respect to $G{\otimes}_O K$ if and only if 
$\beta_n\in\frak{l}( K)$ is smoothly regular with respect to $L$. 

Now fix a system of
positive roots $\Phi^+$ in the root system $\Phi(T,L)$ of
$L$ with respect to $T$ such that 
\begin{equation}
 \beta_n=\sum_{\alpha\in\Phi^+}c_{\alpha}\cdot X_{\alpha}
\label{eq:root-expension-of-nilpotent-part-of-beta}
\end{equation}
where $X_{\alpha}$ is a root vector of the root $\alpha$. Let us
denote by $\Gamma$ (resp. $\Gamma^{\prime}$) the root lattice (the
weight lattice) of the root system $\Phi(T,L)$. Then the finite group
$\Gamma^{\prime}/\Gamma$ is called the fundamental group of
$\Phi(T,L)$ (see \cite[(1.4)]{Springer1966}). 
Considering the
decomposition into the simple factors, we may assume that the root
system $\Phi(T,L)$ is simple. Then results of 
\cite[(5.8), (5.9)]{Springer1966} and 
\cite[p.228,III-3.5]{Borel1968} implies 

\begin{prop}
\label{prop:regularity-in-terms-of-coefficient-of-root-vector}
Assume that $G$ is semi-simple, that is, the fibers 
$G{\otimes}_OK$ ($K=F,\Bbb F$) are semi-simple algebraic $K$-groups. 
Assume also that the characteristic of $K$ is not the bad prime and does
not divide the order of the fundamental group of 
$\Phi(T,L)$ tabulated below. 
\begin{center}
\begin{tabular}{c|c|c|c|c|c|c|c}
type of $\Phi(T,L)$   
         &$A_m$      &$B_m, C_m$&$D_m$&$E_6$&$E_7$&$E_8$  &$F_4, G_2$\\
\hline
bad prime&$\emptyset$&   $2$    &$2$  &$2,3$&$2,3$&$2,3,5$&$2,3$\\
\hline
oder of $\Gamma^{\prime}/\Gamma$
         &  $m+1$    &   $2$    & $4$ & $3$ & $2$ & $1$   & $1$.
\end{tabular}
\end{center}
Then the following three statements are equivalent:
\begin{enumerate}
\item $\overline\beta\in\frak{g}(K)$ is smoothly regular with respect
      to $G$ over $K$,
\item $\overline\beta\in\frak{g}(K)$ is regular with respect
      to $G$ over $K$,
\item $c_{\alpha}\neq 0$ for all simple $\alpha\in\Phi^+$ in 
      \eqref{eq:root-expension-of-nilpotent-part-of-beta}.
\end{enumerate}
\end{prop}

\begin{rem}
\label{remark:smooth-regularity-and-connectedness-of-centralizer}
Assume that $\overline\beta\in\frak{g}(K)$ is smoothly regular with
respect to $G{\otimes}_OK$. Then $G_{\beta}{\otimes}_OK$ is connected
if $Z_{G{\otimes}_OK}(\beta_s)$ and its center are connected 
(see Theorem 5.9 b) of \cite{Springer1966}).
\end{rem}

\subsection[]{Case of $GL_n$}
\label{subsec:detailed-description-of-the-case-gl-n}
Let us consider the case of $GL_n$ ($n\geq 2$) which is a connected smooth
reductive $O$-group scheme. For $\beta\in\frak{gl}_n(O)$, the
following statements are equivalent ($K=F,\Bbb F$);
\begin{enumerate}
\item $c_{\alpha}\neq 0$ for all simple $\alpha\in\Phi^+$ in 
      \eqref{eq:root-expension-of-nilpotent-part-of-beta},
\item the characteristic polynomial 
      $\chi_{\overline\beta}(t)
       =\det(t1_n-\overline\beta)\in K[t]$ is the
      minimal polynomial of $\overline\beta\in M_n(K)$,
\item $\overline\beta\in M_n(K)$ is 
      $GL_n(\overline K)$-conjugate to 
\begin{equation}
 J_{n_1}(\alpha_1)\boxplus\cdots\boxplus J_{n_r}(\alpha_r)=
 \begin{bmatrix}
  J_{n_1}(\alpha_1)&      &                 \\
                   &\ddots&                 \\
                   &      &J_{n_r}(\alpha_r)
 \end{bmatrix}
\label{eq:jordan-decomposition-of-beta}
\end{equation}
      where $\alpha_1,\cdots,\alpha_r$ are distinct elements of the
      algebraic closure $\overline K$ of $K$ and
$$
 J_m(\alpha)=\begin{bmatrix}
              \alpha&   1  &      &      \\
                    &\alpha&\ddots&      \\
                    &      &\ddots&   1  \\
                    &      &      &\alpha\\
             \end{bmatrix}
$$
      is a Jordan block of size $m$,
\item $\{X\in M_n(K)\mid X\overline\beta=\overline\beta X\}
       =K[\overline\beta]$.
\end{enumerate}
In these cases, $\overline\beta\in\frak{gl}_n(K)$ is smoothly regular
with respect to $GL_n$ over $K$ and 
the fiber $GL_{n,\beta}{\otimes}_OK$ is connected. 

Note also that the following statements are equivalent;
\begin{enumerate}
\item the characteristic polynomial of $\overline\beta\in M_n(\Bbb F)$
      is its minimal polynomial,
\item $\{X\in M_n(O_l)\mid X\beta\equiv\beta X\pmod{\frak{p}^l}\}
       =O_l[\beta_l]$ for all $l>0$,
       with $\beta_l=\beta\pmod{\frak{p}^l}$ 
\item $O^n$ is a cyclic $O[\beta]$-module, that is, there exists a
      vector $v\in O^n$ such that $O^n=O[\beta]v$.
\end{enumerate}
In these cases we have
$$
 \{X\in M_n(O)\mid X\beta=\beta X\}=O[\beta]
$$
and the characteristic polynomial of $\beta\in M_n(O)$ is its minimal
polynomial. In particular $GL_{n,\beta}$ is a smooth over $O$ 
by Proposition
\ref{tprop:sufficient-condition-for-smooth-commutativeness-of-g-beta}. 

It is easy to show that the conditions I), II) and III) of section 
\ref{subsec:fundamental-setting} hold for $GL_n$. So 
Theorem \ref{th:main-result} gives

\begin{thm}\label{th:main-result-for-gl-n}
Take a $\beta\in\frak{gl}_n(O)$ such that the characteristic
polynomial of $\overline\beta\in M_n(\Bbb F)$ is its minimal
polynomial. Then we have a bijection 
$\theta\mapsto
 \text{\rm Ind}_{GL_n(O_r,\beta)}^{GL_n(O_r)}\sigma_{\beta,\theta}$ 
of the set 
$$
  \left\{\begin{array}{l}
   \theta:O_r[\beta_r]^{\times}\to\Bbb C^1
     :\text{\rm group homomorphism}\\
   \hphantom{\theta:}
        \text{\rm s.t. 
  $\theta(x\npmod{\frak{p}^r})
    =\tau\left(\varpi^{-r}\text{\rm tr}(\beta(x-1_n))\right)$}\\
   \hphantom{\theta:s.t.}
        \text{\rm for 
   $\forall x\in 1_n+\varpi^lO[\beta]$}
         \end{array}\right\}
$$
onto the set 
$\text{\rm Irr}(GL_n(O_r)\mid\psi_{\beta})$.
\end{thm}

In order to give a more explicit description, let us recall the
following results due to Shintani \cite{Shintani1968}.

Let $L/F$ be a tamely ramified  separable field extension of degree $n$ 
and $O_L\subset L$ the integer ring with the maximal ideal
$\frak{p}_L=\varpi_LO_L$. The residue class field $\Bbb F=O/\frak{p}$
is identified with a subfield of $\Bbb L=O_L/\frak{p}_L$. 
A prime element $\varpi_L$ can be chosen so
that we have $\varpi_L^e\in O_{L_0}$ where  $L_0$ is the maximal
unramified subextension of $L/F$ and $e=(L:L_0)$ is the ramification
index of $L/F$. Then we have $O_L=O_{L_0}[\varpi_L]$. 
We will identify $L$ with a $F$-subalgebra of $M_n(F)$ by means of
the regular representation of $L$ with respect to an $O$-basis of
$O_L$. Then we have
\cite[p.545, Lemma 4-7, Cor.1, p.546, Cor.2]{Shintani1968}

\begin{prop}
\label{prop:generator-of-tame-extension-is-smoothly-regular}
For a $\beta=\sum_{i=0}^{e-1}b_i\varpi_L^i\in O_L$ with 
$b_i\in O_{L_0}$, the following two statements are equivalent;
\begin{enumerate}
\item $O_L=O[\beta]$,
\item $\overline{b_0}^{\sigma}\neq\overline{b_0}$ for all 
      $1\neq\sigma\in\text{\rm Gal}(\Bbb L/\Bbb F)$, and 
      $b_1\in O_{L_0}^{\times}$ if $e>1$.
\end{enumerate}
In this case, the characteristic
polynomial $\chi_{\beta}(t)=\det(t\cdot 1_n-\beta)$ of 
$\beta\in M_n(O)$ has the following properties;
\begin{enumerate}
\item $\chi_{\beta}(t)\npmod{\frak p}\in\Bbb F[t]$ is the minimal
  polynomial of $\overline\beta\in M_n(\Bbb F)$,
\item $\chi_{\beta}(t)\npmod{\frak p}=p(t)^e$ with an irreducible
      polynomial $p(t)\in\Bbb F[t]$, 
\item $\chi_{\beta}(t)\npmod{\frak{p}^2}\in O_2[t]$ is irreducible.
\end{enumerate}
\end{prop}

So any $\beta\in O_L$ such that $O_L=O[\beta]$ gives an
  example of a $\beta\in\frak{gl}_n(O)$ which is smoothly regular with
  respect to $G=GL_n$. Then
Theorem \ref{th:main-result} gives the following result of 
Shintani \cite[Prop.4-2, Prop.4-3]{Shintani1968}

\begin{thm}
\label{th:main-result-for-gl-n-related-with-tamely-ramified-extension}
There is a bijection 
$\theta\mapsto
 \text{\rm Ind}_{GL_n(O_r,\beta)}^{GL_n(O_r)}\sigma_{\beta,\theta}$ 
of the set 
$$
  \left\{\begin{array}{l}
        \theta:\left(O_L/\frak{p}_L^{er}\right)^{\times}
               \to\Bbb C^{\times}:\text{\rm group homomorphism}\\
   \hphantom{\theta:}
        \text{\rm s.t. 
  $\theta(x\npmod{\frak{p}_L^{er}})
    =\tau\left(\varpi^{-r}T_{L/F}(\beta (x-1))\right)$}\\
   \hphantom{\theta:s.t.}
        \text{\rm for 
   $\forall x\in 1+\frak{p}_L^{el}$}
         \end{array}\right\}
$$
onto the set $\text{\rm Irr}(GL_n(O_r)\mid\psi_{\beta})$.
\end{thm}


\subsection[]{Regularity for classical groups}
\label{subsec:smooth-regularity-for-classical-group}
Let us see the smooth regularity over $K=F$ or $\Bbb F$ 
for the special linear group, the
symplectic group and the orthogonal group. 
Explicit applications, as Theorem 
\ref{th:main-result-for-gl-n-related-with-tamely-ramified-extension} 
for $GL_n$,
of Theorem \ref{th:main-result} to these groups are given in section 
\ref{sec:classical-group}.

\subsubsection[]{Case of $SL_n$}
\label{subsubsec:smooth-regularity-for-sl-n}
For the algebraic group $SL_n$ over $K$, the situation is almost the
same as the case of $GL_n$ in subsection
\ref{subsec:detailed-description-of-the-case-gl-n}, except that 
$SL_{n,\beta}{\otimes}_OK$ is connected if and only if the characteristic
polynomial of $\overline\beta\in M_n(K)$ is separable, that is 
$n_1=\cdots=n_r=1$ in \eqref{eq:jordan-decomposition-of-beta}. 

\subsubsection[]{Case of $Sp_{2n}$}
\label{subsubsec:smooth-regularity-for-sp-2n}
The algebraic group $Sp_{2n}$ over $K$ is defined by
$$
 Sp_{2n}=\{g\in GL_{2n}\mid gJ_n\,^tg=J_n\}
$$
where
$$
 J_n=\begin{bmatrix}
      0&I_n\\
     -I_n&0
     \end{bmatrix}
 \quad
 \text{\rm with}
 \quad
 I_n=\begin{bmatrix}
       &       &1\\
       &\addots& \\
      1&       &
     \end{bmatrix}.
$$
Then a maximal torus $T$ of $Sp_{2n}$ is
$$
 T=\left\{\begin{bmatrix}
           a&0\\
           0&^{\tau}a^{-1}
          \end{bmatrix}\biggm| a=\begin{bmatrix}
                                  a_1&      &   \\
                                     &\ddots&   \\
                                     &      &a_n
                                 \end{bmatrix}\right\}
$$
where $^{\tau}a=I_n\,^taI_n$. The Weyl group is generated by the
permutations among $\{a_1,\cdots,a_n\}$ and the sign changes 
$a_i\mapsto a_i^{\pm 1}$. Any maximal torus of $Sp_{2n}$ is 
$Sp_{2n}(\overline K)$-conjugate to $T$  ($\overline K$ is the
algebraic closure of $K$). 
Then for any $\beta\in\text{\rm Lie}(Sp_{2n})(K)$, the
semi-simple part $\beta_s$ of $\beta\in M_{2n}(K)$ is 
$Sp_{2n}(\overline K)$-conjugate to an element
$$
 \begin{bmatrix}
  A&0\\
  0&-^{\tau}A
 \end{bmatrix}
 \quad
 \text{\rm with}
 \quad
 A=\begin{bmatrix}
    a_11_{n_1}&      &         \\
              &\ddots&         \\
              &      &a_r1_{n_r}
   \end{bmatrix}
$$
of $\text{\rm Lie}(T)(\overline K)$, where $a_i\neq\pm a_j$ if 
$i\neq j$ and $a_i\neq 0$ for $1\leq i<r$ possibly $a_r=0$. Then 
the centralizer $L=Z_{Sp_{2n}}(\beta_s)^o$ is 
$$
 L=GL_{n_1}\times\cdots\times GL_{n_{r-1}}\times
   \begin{cases}
    GL_{n_r}&:a_r\neq 0,\\
    Sp_{2n_r}&:a_r=0.
   \end{cases}
$$
We can choose a system of positive roots so that the
nilpotent part $\beta_n$ of $\beta\in M_{2n}(K)$ is an upper
triangle matrix. Then Proposition
\ref{prop:regularity-in-terms-of-coefficient-of-root-vector} says that 
the following two statements are equivalent:
\begin{enumerate}
\item $\beta$ is smoothly regular with respect to $Sp_{2n}$ over $K$, 
\item the characteristic polynomial of $\beta\in M_{2n}(K)$
      is its minimal polynomial.
\end{enumerate}
Note that the centralizer $Sp_{2n,\beta}$ is connected
if and only if $\det\beta\neq 0$.

\subsubsection[]{Case of $SO_{2n+1}$}
\label{subsubsec:smooth-regularity-for-so-an+1}
The algebraic group $SO_{2n+1}$ over $K$ is defined by
$$
 SO_{2n+1}=\{g\in SL_{2n+1}\mid gS\,^tg=S\}
$$
where
$$
 S=\begin{bmatrix}
      & &I_n\\
      &1&   \\
   I_n& &
   \end{bmatrix}
 \quad
 \text{\rm with}
 \quad
 I_n=\begin{bmatrix}
       &       &1\\
       &\addots& \\
      1&       &
     \end{bmatrix}.
$$
Then a maximal torus $T$ of $SO_{2n+1}$ is
$$
 T=\left\{\begin{bmatrix}
           a& &   \\
            &1&   \\
            & &^{\tau}a^{-1}
          \end{bmatrix}\biggm| a=\begin{bmatrix}
                                  a_1&      &   \\
                                     &\ddots&   \\
                                     &      &a_n
                                 \end{bmatrix}\right\}
$$
where $^{\tau}a=I_n\,^taI_n$. The Weyl group is generated by the
permutations among $\{a_1,\cdots,a_n\}$ and the sign changes 
$a_i\mapsto a_i^{\pm 1}$. Any maximal torus of $SO_{2n+1}$ is 
$SO_{2n+1}(\overline K)$-conjugate to $T$  ($\overline K$ is the
algebraic closure of $K$). Then for any 
$\beta\in\text{\rm Lie}(SO_{2n+1})(K)$, the 
semi-simple part $\beta_s$ of $\beta\in M_{2n+1}(K)$ is 
$SO_{2n+1}(\overline K)$-conjugate to an element
$$
 \begin{bmatrix}
  A& &  \\
   &0&  \\
   & &-^{\tau}A
 \end{bmatrix}
 \quad
 \text{\rm with}
 \quad
 A=\begin{bmatrix}
    a_11_{n_1}&      &         \\
              &\ddots&         \\
              &      &a_r1_{n_r}
   \end{bmatrix}
$$
of $\text{\rm Lie}(T)(\overline K)$, where $a_i\neq\pm a_j$ if 
$i\neq j$ and $a_i\neq 0$ for $1\leq i<r$ possibly $a_r=0$. Then 
the centralizer $L=Z_{SO_{2n+1}}(\beta_s)^o$ is 
$$
 L=GL_{n_1}\times\cdots\times GL_{n_{r-1}}\times
   \begin{cases}
    GL_{n_r}&:a_r\neq 0,\\
    SO_{2n_r+1}&:a_r=0.
   \end{cases}
$$
We can choose a system of positive roots so that the
nilpotent part $\beta_n$ of $\beta\in M_{2n+1}(K)$ is an upper
triangle matrix. Then Proposition
\ref{prop:regularity-in-terms-of-coefficient-of-root-vector} says that 
the following two statements are equivalent:
\begin{enumerate}
\item $\beta$ is smoothly regular with respect to $SO_{2n+1}$ over
      $K$, 
\item the characteristic polynomial of $\beta\in M_{2n+1}(K)$
      is its minimal polynomial.
\end{enumerate}
In this case the centralizer $SO_{2n+1,\beta}$ is connected. 

\subsubsection[]{Case of $SO_{2n}$}
\label{subsubsec:smooth-regularity-for-so-2n}
The algebraic group $SO_{2n}$ over $K$ is defined by
$$
 SO_{2n}=\{g\in SL_{2n}\mid gS\,^tg=S\}
$$
where
$$
 S=\begin{bmatrix}
      &I_n\\
   I_n&
   \end{bmatrix}
 \quad
 \text{\rm with}
 \quad
 I_n=\begin{bmatrix}
       &       &1\\
       &\addots& \\
      1&       &
     \end{bmatrix}.
$$
Then a maximal torus $T$ of $SO_{2n}$ is
$$
 T=\left\{\begin{bmatrix}
           a&   \\
            &^{\tau}a^{-1}
          \end{bmatrix}\biggm| a=\begin{bmatrix}
                                  a_1&      &   \\
                                     &\ddots&   \\
                                     &      &a_n
                                 \end{bmatrix}\right\}
$$
where $^{\tau}a=I_n\,^taI_n$. The Weyl group is generated by the
permutations among $\{a_1,\cdots,a_n\}$ and the sign changes 
$$
 (a_1,\cdots,a_n)\mapsto
 (a_1^{\varepsilon_1},\cdots,a_n^{\varepsilon_n})
$$
($\varepsilon_i=\pm 1$) such that 
$\varepsilon_1\cdots\varepsilon_n=1$.
Any maximal torus of $SO_{2n}$ is 
$SO_{2n}(\overline K)$-conjugate to $T$  ($\overline K$ is the
algebraic closure of $K$). Then for any 
$\beta\in\text{\rm Lie}(SO_{2n})(K)$, the 
semi-simple part $\beta_s$ of $\beta\in M_{2n}(K)$ is 
$SO_{2n}(\overline K)$-conjugate to an element
$$
 \begin{bmatrix}
  A&  \\
   &-^{\tau}A
 \end{bmatrix}
 \quad
 \text{\rm with}
 \quad
 A=\begin{bmatrix}
    a_11_{n_1}&      &         \\
              &\ddots&         \\
              &      &a_r1_{n_r}
   \end{bmatrix}
$$
of $\text{\rm Lie}(T)(\overline K)$, where $a_i\neq a_j$ if 
$i\neq j$ and $a_i\neq 0$ for $1\leq i<r$ possibly $a_r=0$. If
$a_i\neq\pm a_j$ for all $i\neq j$ and $a_r\neq 0$, then 
the centralizer $L=Z_{SO_{2n}}(\beta_s)^o$ is 
$$
 L=GL_{n_1}\times\cdots GL_{n_r}.
$$
If $a_i=-a_j$ for some $i\neq j$, then $L$ contains the factor 
$GL_{n_1+n_j}$. If $a_r=0$, then $L$ contains the factor 
$SO_{2n_r}$ which breaks our argument. For example
$$
 \begin{bmatrix}
  0&1&1&0\\
   &0&0&-1\\
   & &0&-1\\
   & & &0
 \end{bmatrix}
$$
is smoothly regular with respect to $SO_4$ over $K$ but the
characteristic polynomial is not the minimal polynomial. 
Any way we can choose a system of positive roots so that the
nilpotent part $\beta_n$ of $\beta\in M_{2n}(K)$ is an upper
triangle matrix. If $\det\beta\neq 0$ then Proposition
\ref{prop:regularity-in-terms-of-coefficient-of-root-vector} says that 
the following two statements are equivalent:
\begin{enumerate}
\item $\beta$ is smoothly regular with respect to $SO_{2n}$ over
      $K$, 
\item the characteristic polynomial of $\beta\in M_{2n}(K)$
      is its minimal polynomial.
\end{enumerate}
In this case the centralizer $SO_{2n,\beta}$ is connected.


\subsection[]{Some results of Shechter \cite{Shechter2019}}
\label{subsec:discussionon-shechter-results}
Let $G$ be a connected reductive $O$-group scheme which satisfies the
conditions I), II) and III) of subsection 
\ref{subsec:fundamental-setting}. 
Take an adjoint $G(\Bbb F)$-orbit $\Omega\subset\frak{g}(\Bbb F)$
whose elements are smoothly regular with respect to $G$ over $\Bbb
F$. Take a  
$\beta\in\frak{g}(O)$ such that $\overline\beta\in\Omega$ and assume
that the two conditions of Theorem \ref{th:main-result} are satisfied
with respect to this $\beta$. Since we have canonical isomorphisms 
$$
 \frak{g}(\Bbb F)\,\tilde{\to}\,K_{m-1}(O_m),
 \quad
 \frak{g}_{\beta}(\Bbb F)\,\tilde{\to}\,
  G_{\beta}(O_m)\cap K_{m-1}(O_m)
$$
for any $m>1$, we have
$$
 |G(O_l)|=|G(\Bbb F)|\cdot q^{(l-1)\dim G},
 \qquad
 |G_{\beta}(O_l)|=|G_{\beta}(\Bbb F)|\cdot q^{(l-1)\text{\rm rank}\,G}
$$
for all $l>0$, where $\text{\rm rank}\,G$ is the absolute rank of $G$.  
Since the canonical group homomorphism 
$G(O_{l^{\prime}})\to G(\Bbb F)$ is surjective, the inverse image of 
$\Omega$ under the canonical surjection 
$\frak{g}(O_{l^{\prime}})\to\frak{g}(\Bbb F)$ is divided into 
$q^{(l^{\prime}-1)\text{\rm rank}\,G}$ adjoint 
$G(O_{l^{\prime}})$-orbits consisting of 
$$
 \frac{|G(O_{l^{\prime}})|}
      {|G_{\beta}(O_{l^{\prime}})|}
 =\sharp\Omega\cdot 
  q^{(l^{\prime}-1)(\dim G-\text{\rm rank}\,G)}.
$$
elements. On the other hand, we have
\begin{align*}
 &\sharp\left\{\theta\in G_{\beta}(O_r)\sphat\;\;\;
          \text{\rm s.t. $\theta=\psi_{\beta}$ on 
                $G_{\beta}(O_r)\cap K_l(O_r)$}\right\}\\
 =&\left(G_{\beta}(O_r):G_{\beta}(O_r)\cap K_l(O_r)\right)
 =|G_{\beta}(O_l)|\\
 =&|G_{\beta}(\Bbb F)|\cdot q^{(l-1)\text{\rm rank}\,G}
 =\frac{|G(\Bbb F)|}
       {\sharp\Omega}\cdot q^{(l-1)\text{\rm rank}\,G}.
\end{align*}
Then Theorem \ref{th:main-result} implies that the number of the
irreducible representations of $G(O_r)$ corresponding to the adjoint
orbit $\Omega\subset\frak{g}(\Bbb F)$ is 
$$
 \frac{|G(\Bbb F)|}
       {\sharp\Omega}\cdot q^{(l-1)\text{\rm rank}\,G}\times 
 q^{(l^{\prime}-1)\text{\rm rank}\,G}
 =\frac{|G(\Bbb F)|}
       {\sharp\Omega}\cdot q^{(r-2)\text{\rm rank}\,G}.
$$
This is the first statement of Theorem I of Shechter
\cite{Shechter2019}. This means that we have constructed the
irreducible representations of Shechter \cite{Shechter2019}. 

The construction of $\sigma_{\beta,\theta}$ presented in subsection 
\ref{subsec:proof-of-main-result-for-even-r} (when $r$ is even) and 
section \ref{sec:schur-multiplier} (when $r$ is odd) shows that
$$
 \dim\sigma_{\beta,\theta}
 =\begin{cases}
   1&:\text{\rm $r$ is even},\\
   q^{\frac 12\dim_{\Bbb F}(\frak{g}(\Bbb F)/\frak{g}_{\beta}(\Bbb F))}
   =q^{(\dim G-\text{\rm rank}\,G)/2}
     &:\text{\rm $r$ is odd}.
  \end{cases}
$$
On the other hand, we have
\begin{align*}
 \left(G(O_r):G(O_r,\beta)\right)
 &=\frac{|G(O_{l^{\prime}})|}
        {|G_{\beta}(O_{l^{\prime}})|}\\
 &=\sharp\Omega\cdot q^{(l^{\prime}-1)(\dim G-\text{\rm rank}\,G)},
\end{align*}
because we have $G(O_r,\beta)=G_{\beta}(O_r)\cdot K_{l^{\prime}}(O_r)$
and
$$
 |K_{l^{\prime}}(O_r)|=\frac{|G(O_r)|}
                            {|G(O_{l^{\prime}})|},
 \quad
 |G_{\beta}(O_r)\cap K_{l^{\prime}}(O_r)|
 =\frac{|G_{\beta}(O_r)|}
       {|G_{\beta}(O_{l^{\prime}})|}.
$$
Then we have
\begin{align*}
 \dim\delta_{\beta,\theta}
 &=\left(G(O_r):G(O_r,\beta)\right)\cdot\dim\sigma_{\beta,\theta}\\
 &=\sharp\Omega\cdot q^{(r-2)(\dim G-\text{\rm rank}\,G)/2}
\end{align*}
which is the second statement of Theorem I of Shechter
\cite{Shechter2019}. 

%% file: schur.tex
\section{Schur multiplier}
\label{sec:schur-multiplier}
Let $G\subset GL_n$ be a closed $\Bbb F$-algebraic subgroup and 
$\frak{g}$ the Lie algebra scheme of $G$ which is a closed affine 
$\Bbb F$-subscheme of the Lie algebra scheme $\frak{gl}_n$ of
$GL_n$. Let us assume that the trace form
$$
 B:\frak{g}(\Bbb F)\times\frak{g}(\Bbb F)\to\Bbb F
 \quad
 ((X,Y)\mapsto\text{\rm tr}(XY))
$$
is non-degenerate. Fix a $\beta\in\frak{g}(\Bbb F)$ such that 
$\frak{g}_{\beta}(\Bbb F)\lvertneqq\frak{g}(\Bbb F)$.

We have fixed a continuous unitary character $\tau$ of the additive
group $F$ such that 
$$
 \{x\in F\mid\tau(xO)=1\}=O.
$$
Define an additive character $\overline\tau$ of $\Bbb F$ by 
$\overline\tau(\overline x)=\tau(\varpi^{-1}x)$. 

\subsection[]{Definition of a Schur multiplier}
\label{subsec:definition-of-schur-multiplier}
The non-zero $\Bbb F$-vector space 
$\Bbb V_{\beta}=\frak{g}(\Bbb F)/\frak{g}_{\beta}(\Bbb F)$ has a
symplectic form 
$$
 \langle \dot X,\dot Y\rangle_{\beta}=B([X,Y],\overline\beta)
$$
where $\dot X=X\pmod{\frak{g}_{\beta}(\Bbb F)}\in\Bbb V_{\beta}$ with 
$X\in\frak{g}_{\beta}(\Bbb F)$.  
Then $g\in G_{\beta}(\Bbb F)$ gives an element $\sigma_g$ of 
 the symplectic group $Sp(\Bbb V_{\beta})$ defined by 
$$
 X\npmod{\frak{g}_{\beta}(\Bbb F)}\mapsto
 \text{\rm Ad}(g)^{-1}X\npmod{\frak{g}_{\beta}(\Bbb F)}.
$$
Note that the group $Sp(\Bbb V_{\beta})$ acts on $\Bbb V_{\beta}$ from
right. Let 
$v\mapsto[v]$ be a $\Bbb F$-linear section on $\Bbb V_{\beta}$ of the
exact sequence 
\begin{equation}
 0\to\frak{g}_{\beta}(\Bbb F)
  \to\frak{g}(\Bbb F)\to\Bbb V_{\beta}\to 0.
\label{eq:fundamental-exact-sequence-of-lie-alg-and-symplectic-space}
\end{equation}
For any $v\in\Bbb V_{\beta}$ and $g\in G_{\beta}(\Bbb F)$, put
$$
 \gamma(v,g)=\gamma_{\frak g}(v,g)
 =\text{\rm Ad}(g)^{-1}[v]-[v\sigma_g]\in\frak{g}_{\beta}(\Bbb F).
$$
Take a character $\rho\in\frak{g}_{\beta}(\Bbb F)\sphat$. Then there
exists uniquely a $v_g\in\Bbb V_{\beta}$ such that
$$
 \rho(\gamma(v,g))=\overline\tau(\langle v,v_g\rangle_{\beta})
$$
for all $v\in\Bbb V_{\beta}$. Note that $v_g\in\Bbb V_{\beta}$ depends
on $\rho$ as well as the section $v\mapsto[v]$. Let
$$
 G_{\beta}(\Bbb F)^{(c)}
 =\{g\in G(\Bbb F)\mid\text{\rm Ad}(g)Y=Y\;\text{\rm for}\;
                        \forall Y\in\frak{g}_{\beta}(\Bbb F)\}
$$
be the centralizer of $\frak{g}_{\beta}(\Bbb F)$ in $G(\Bbb F)$, which
is a subgroup of $G_{\beta}(\Bbb F)$. Then for any 
$g,h\in G_{\beta}(\Bbb F)^{(c)}$, we have
\begin{equation}
 v_{gh}=v_h\sigma_g^{-1}+v_g
\label{eq:multiplication-formula-of-v-epsilon}
\end{equation}
because $\gamma(v,gh)=\gamma(v,g)+\gamma(v\sigma_g^{-1},h)$ for all
$v\in\Bbb V_{\beta}$. Put
$$
 c_{\beta,\rho}(g,h)
 =\overline\tau(2^{-1}\langle v_g,v_{gh}\rangle_{\beta})
$$
for $g,h\in G_{\beta}(\Bbb F)^{(c)}$. Then the relation 
\eqref{eq:multiplication-formula-of-v-epsilon} shows that 
$c_{\beta,\rho}\in Z^2(G_{\beta}(\Bbb F)^{(c)},\Bbb C^{\times})$ is a
2-cocycle with trivial action of $G_{\beta}(\Bbb F)^{(c)}$ on 
$\Bbb C^{\times}$. Moreover we have

\begin{prop}
\label{prop:schur-multiplier-associated-with-finite-symplectic-space}
The cohomology class 
$[c_{\beta,\rho}]\in H^2(G_{\beta}(\Bbb F)^{(c)},\Bbb C^{\times})$ is
independent of the choice of the $\Bbb F$-linear section 
$v\mapsto[v]$.
\end{prop}
\begin{proof}
Take another $\Bbb F$-linear section $v\mapsto[v]^{\prime}$ 
with respect to which
we will define $\gamma^{\prime}(v,g)\in\frak{g}_{\beta}$ 
and $v^{\prime}_g\in\Bbb V_{\beta}$ as above. 
Then there exists a $\delta\in\Bbb V_{\beta}$ such that 
$\rho([v]-[v]^{\prime})
 =\overline\tau\left(\langle v,\delta\rangle_{\beta}\right)$ 
for all $v\in\Bbb V_{\beta}$. We have 
$v_g^{\prime}=v_g+\delta-\delta\sigma_g$ for all 
$g\in G_{\beta}(\Bbb F)^{(c)}$. So if we put 
$\alpha(g)
 =\overline\tau\left(2^{-1}\langle 
    v_g^{\prime}-v_{g^{-1}},\delta\rangle_{\beta}
            \right)$ 
for $g\in G_{\beta}(\Bbb F)^{(c)}$, then we have
$$
 \overline\tau\left(2^{-1}\langle v_g^{\prime},
           v_{gh}^{\prime}\rangle_{\beta}\right)
 =\overline\tau\left(2^{-1}\langle v_g,
           v_{gh}\rangle_{\beta}\right)\cdot
  \alpha(h)\alpha(gh)^{-1}\alpha(g)
$$
for all $g,h\in G_{\beta}(\Bbb F)^{(c)}$.
\end{proof}

\subsection[]{Relation to an over group}
\label{subsec:schur-multiplier-and-over-group}
Let us assume that there exists a closed smooth $O$-group subscheme 
$H\subset GL_n$ of which our $G$ is a closed $O$-group subscheme and
that the trace form
$$
 B:\frak{h}(\Bbb F)\times\frak{h}(\Bbb F)\to\Bbb F
$$ 
is non-degenerate where $\frak{h}$ is the Lie algebra scheme of $H$. 
Then we have 
$$
 \frak{h}(\Bbb F)=\frak{g}(\Bbb F)\oplus\frak{g}(\Bbb F)^{\perp}
$$
where 
$\frak{g}(\Bbb F)^{\perp}
 =\{X\in\frak{h}(\Bbb F)\mid B(X,\frak{g}(\Bbb F))=0\}$ 
is the orthogonal complement of $\frak{g}(\Bbb F)$ in 
$\frak{h}(\Bbb F)$. 

Take a $\beta\in\frak{g}(O)$ such that 
$\frak{g}_{\beta}(\Bbb F)\lvertneqq\frak{g}(\Bbb F)$. Then 
$\beta\in\frak{h}(O)$ and 
$\frak{h}_{\beta}(\Bbb F)\lvertneqq\frak{h}(\Bbb F)$ where 
$\frak{h}_{\beta}=Z_{\frak h}(\beta)$ is the centralizer. We have
decompositions 
$$
 \frak{h}_{\beta}(\Bbb F)
 =\frak{g}_{\beta}(\Bbb F)\oplus
  \left(\frak{g}(\Bbb F)^{\perp}\right)_{\beta}
$$
where 
$\left(\frak{g}(\Bbb F)^{\perp}\right)_{\beta}
 =\frak{h}_{\beta}(\Bbb F)\cap\frak{g}(\Bbb F)^{\perp}$, and
$$
 \widetilde{\Bbb V}_{\beta}
 =\frak{h}(\Bbb F)/\frak{h}_{\beta}(\Bbb F)
 =\Bbb V_{\beta}\oplus
  \left(\frak{g}(\Bbb F)^{\perp}/
   \left(\frak{g}(\Bbb F)^{\perp}\right)_{\beta}\right)
$$
is an orthogonal decomposition of symplectic spaces. 

Let $v\mapsto[v]$ be a $\Bbb F$-linear section of the exact sequence 
$$
 0\to\frak{h}_{\beta}(\Bbb F)\to\frak{h}(\Bbb F)
  \to\widetilde{\Bbb V}_{\beta}\to 0
$$
of $\Bbb F$-vector space such that 
$[\Bbb V_{\beta}]\subset\frak{g}(\Bbb F)$
and  
$[\frak{g}(\Bbb F)^{\perp}/\left(\frak{g}(\Bbb F)^{\perp}\right)_{\beta}]
 \subset\frak{g}(\Bbb F)^{\perp}$. 

Take $\rho\in\frak{g}_{\beta}(\Bbb F)\sphat$ and put
$$
 \widetilde\rho:\frak{h}_{\beta}(\Bbb F)
 =\frak{g}_{\beta}(\Bbb F)\oplus
    \left(\frak{g}(\Bbb F)^{\perp}\right)_{\beta}
 \xrightarrow{\text{\rm projection}}\frak{g}_{\beta}(\Bbb F)
 \xrightarrow{\rho}\Bbb C^{\times}.
$$
For any $g\in G_{\beta}(\Bbb F)\subset H_{\beta}(\Bbb F)$, there
exists uniquely a $v_g\in\Bbb V_{\beta}$ such that
$$
 \rho\left(\gamma_{\frak g}(v,g)\right)
 =\overline\tau\left(\langle v,v_g\rangle_{\beta}\right)
$$
for all $v\in\Bbb V_{\beta}$. Then we have
$$
 \widetilde\rho\left(\gamma_{\frak{h}}(v,v_g)\right)
 =\overline\tau\left(\langle v,v_g\rangle_{\beta}\right)
$$
for all $v\in\widetilde{\Bbb V}_{\beta}$. In fact if we put 
$v=v^{\prime}+v^{\prime\prime}$ with $v^{\prime}\in\Bbb V_{\beta}$ and 
$v^{\prime\prime}\in
 \frak{g}(\Bbb F)^{\perp}/\left(\frak{g}(\Bbb F)^{\perp}\right)_{\beta}$,
then we have 
$\gamma_{\frak{h}}(v,g)
 =\gamma_{\frak g}(v^{\prime},g)
 +\gamma_{\frak{h}}(v^{\prime\prime},g)$ with 
$\gamma_{\frak{h}}(v^{\prime\prime},g)\in
 \left(\frak{g}(\Bbb F)^{\perp}\right)_{\beta}$, since 
$$
 \text{\rm Ad}(g)\frak{g}(\Bbb F)^{\perp}=\frak{g}(\Bbb F)^{\perp},
 \qquad
 \text{\rm Ad}(g)\left(\frak{g}(\Bbb F)^{\perp}\right)_{\beta}
 =\left(\frak{g}(\Bbb F)^{\perp}\right)_{\beta}.
$$
Then we have
\begin{align*}
 \widetilde\rho\left(\gamma_{\frak{h}}(v,g)\right)
 &=\rho\left(\gamma_{\frak g}(v^{\prime},g)\right)
  =\overline\tau\left(\langle v^{\prime},v_g\rangle_{\beta}\right)\\
 &=\overline\tau\left(\langle v,v_g\rangle_{\beta}\right)
\end{align*}
because $\langle v^{\prime\prime},v_g\rangle_{\beta}=0$. Hence we have

\begin{prop}
\label{prop:regular-schur-multiplier-is-restriction-from-upper-group}
If $G_{\beta}(\Bbb F)^{(c)}\subset H_{\beta}(\Bbb F)^{(c)}$ then 
the Schur multiplier 
      $[c_{\beta,\rho}]
       \in H^2(G_{\beta}(\Bbb F)^{(c)},\Bbb C^{\times})$ is the image
       under the restriction mapping 
$$
 \text{\rm Res}:H^2(H_{\beta}(\Bbb F)^{(c)},\Bbb C^{\times})\to
                H^2(G_{\beta}(\Bbb F)^{(c)},\Bbb C^{\times})
$$
      of the Schur multiplier 
      $[c_{\beta,\widetilde\rho}]
       \in H^2(H_{\beta}(\Bbb F)^{(c)},\Bbb C^{\times})$.
\end{prop}

%% file: weilrep.tex
\section{Proof of Theorem \ref{th:main-result} for odd $r$}
\label{sec:weil-representation}
In this section, we will give a proof of Theorem \ref{th:main-result}
in the case of odd $r$. 

Let $G\subset GL_n$ be a smooth $O$-group scheme which satisfies the
conditions described in subsection 
\ref{subsec:fundamental-setting}. Take a $\beta\in\frak{g}(O)$ such
that the centralizer $G_{\beta}$ is commutative and smooth over $O$. 
Put $r=2l-1$ with an integer $l\geq 2$ and put $l^{\prime}=l-1\geq 1$.

\subsection[]{Construction of irreducible representations}
\label{subsec:crude-weil-representation}
We have a chain of canonical surjections
\begin{equation}
 \heartsuit:K_{l-1}(O_r)\to K_{l-1}(O_{r-1})\,\tilde{\to}\,
        \frak{g}(O_{l-1})\to\frak{g}(\Bbb F)
\label{eq:canonical-surjection-of-k-l-1-to-g(f)}
\end{equation}
defined by
\begin{align*}
 1+\varpi^{l-1}X\nnpmod{\frak{p}^r}
 &\mapsto
 1+\varpi^{l-1}X\nnpmod{\frak{p}^{r-1}}\\
 &\mapsto
  X\nnpmod{\frak{p}^{l-1}}\mapsto
  \overline X=X\npmod{\frak p}.
\end{align*}
Here we use the condition II) of
subsection \ref{subsec:fundamental-setting}. 
Let us denote by $Z(O_r,\beta)$ the inverse image 
under the surjection $\heartsuit$ of $\frak{g}_{\beta}(\Bbb F)$. Then 
$Z(O_r,\beta)$ is a normal subgroup of $K_{l-1}(O_r)$ containing 
$K_l(O_r)$ as the kernel of $\heartsuit$. 

Let us denote by $Y_{\beta}$ the set of the group
homomorphisms $\psi$ of $Z(O_r,\beta)$ to $\Bbb C^{\times}$ 
such that $\psi=\psi_{\beta}$ on $K_l(O_r)$. Then
a bijection of $\frak{g}_{\beta}(\Bbb F)\sphat$ onto $Y_{\beta}$ is
given by
$$
 \rho\mapsto\psi_{\beta,\rho}
            =\widetilde{\psi}_{\beta}\cdot(\rho\circ\heartsuit),
$$
where a group homomorphism 
$\widetilde{\psi}_{\beta}:Z(O_r,\beta)\to\Bbb C^{\times}$ 
is defined by 
$$
 1+\varpi^{l-1}X\npmod{\frak{p}^r}\mapsto
 \tau\left(\varpi^{-l}B(X,\beta)-(2\varpi)^{-1}B(X^2,\beta)\right)
$$
with $\overline X=X\npmod{\frak{p}}\in\frak{g}_{\beta}(\Bbb F)$. 

Take a $\psi\in Y_{\beta}$. 
For two elements 
$$
 x=1+\varpi^{l-1}X\npmod{\frak{p}^r},
 \quad
 y=1+\varpi^{l-1}Y\npmod{\frak{p}^r}
$$ 
of $K_{l-1}(O_r)$, we have 
$x^{-1}=1-\varpi^{l-1}X+2^{-1}\varpi^{2l-2}X^2\npmod{\frak{p}^r}$ so
that we have
$$
 xyx^{-1}y^{-1}
 =1+\varpi^{r-1}[X,Y]\npmod{\frak{p}^r}
 \in K_{r-1}(O_r)\subset K_l(O_r)
$$
and so 
$\psi_{\beta}(xyx^{-1}y^{-1})
 =\tau\left(\varpi^{-1}B(X,\text{\rm ad}(Y)\beta)\right)$. 
Hence we have 
$$
 \psi(xyx^{-1}y^{-1})=\psi_{\beta}(xyx^{-1}y^{-1})=1
$$
for all $x\in K_{l-1}(O_r)$ and $y\in Z(O_r,\beta)$ so
that we can define 
$$
 D_{\psi}:K_{l-1}(O_r)/Z(O_r,\beta)\times K_{l-1}(O_r)/Z(O_r,\beta)
          \to\Bbb C^{\times}
$$
by
$$
 D_{\psi}(\dot g,\dot h)
 =\psi(ghg^{-1}h^{-1})
 =\psi_{\beta}(ghg^{-1}h^{-1})
 =\tau\left(\varpi^{-1}B([X,Y],\beta)\right)
$$
for 
$g=(1+\varpi^{l-1}X)\npmod{\frak{p}^r}, 
 h=(1+\varpi^{l-1}Y)\npmod{\frak{p}^r}
 \in K_{l-1}(O_r)$. Note that $D_{\psi}$ is non-degenerate. 
Then Proposition 3.1.1 of \cite{Takase2016} gives

\begin{prop}
\label{prop:existence-of-fundamental-rep-associated-with-psi} 
For any $\psi=\psi_{\beta,\rho}\in Y_{\beta}$ with 
$\rho\in\frak{g}_{\beta}(\Bbb F)\sphat$, 
there exists unique irreducible
representation $\pi_{\psi}$ of $K_{l-1}(O_r)$ such that 
$\langle\psi,\pi_{\psi}\rangle_{Z(O_r,\beta)}>0$. Furthermore 
$$
 \text{\rm Ind}_{Z(O_r,\beta)}^{K_{l-1}(O_r)}\psi
 =\bigoplus^{\dim\pi_{\psi}}\pi_{\psi}
$$
and $\pi_{\psi}(x)$ is the homothety $\psi(x)$ for all 
$x\in Z(O_r,\beta)$.
\end{prop}

Fix a $\psi=\psi_{\beta,\rho}\in Y_{\beta}$ with 
$\rho\in\frak{g}_{\beta}(\Bbb F)\sphat$. Our problem is to extend the
representation $\pi_{\psi}$ of $K_{l-1}(O_r)$ to a representation of 
$G(O_r,\beta)=G_{\beta}(O_r)\cdot K_{l-1}(O_r)$. 
Now for any $g_r=g\npmod{\frak{p}^r}\in G_{\beta}(O_r)$ and 
$x=(1+\varpi^{l-1}X)\npmod{\frak{p}^r}\in Z(O_r,\beta)$, we have
\begin{align*}
 g_r^{-1}xg_rx^{-1}
 &=\left(1+\varpi^{l-1}g^{-1}Xg\right)
  \left(1-\varpi^{l-1}X+2^{-1}\varpi^{2l-2}X^2\right)\npmod{\frak{p}^r}\\
 &=1+\varpi^{l-1}\left(\text{\rm Ad}(g)^{-1}X-X\right)
    \npmod{\frak{p}^r}
  \in K_l(O_r),
\end{align*}
and
$$
 \psi(g_r^{-1}xg_rx^{-1})
 =\psi_{\beta}(g_r^{-1}xg_rx^{-1})
 =\tau\left(\varpi^{-l}B(X,\text{\rm Ad}(g)\beta-\beta)\right)=1,
$$
that is $\psi(g_r^{-1}xg_r)=\psi(x)$ for all $x\in Z(O_r,\beta)$. This
means that, for any $g\in G_{\beta}(O_r)$, the $g$-conjugate of
$\pi_{\psi}$ is isomorphic to $\pi_{\psi}$, that is, there
exists a $U(g)\in GL_{\Bbb C}(V_{\psi})$ ($V_{\psi}$ is the
representation space of $\pi_{\psi}$) such that
$$
 \pi_{\psi}(g^{-1}xg)
 =U(g)^{-1}\circ\pi_{\psi}(x)\circ U(g)
$$
for all $x\in K_{l-1}(O_r)$, and moreover, for any 
$g,h\in G_{\beta}(O_r)$, there exists a 
$c_U(g,h)\in\Bbb C^{\times}$ such that
$$
 U(g)\circ U(h)=c_U(g,h)\cdot U(gh).
$$
Then $c_U\in Z^2(G_{\beta}(O_r),\Bbb C^{\times})$ is a 
$\Bbb C^{\times}$-valued 2-cocycle on $G_{\beta}(O_r)$ with
trivial action on $\Bbb C^{\times}$, and the cohomology class 
$[c_U]\in H^2(G_{\beta}(O_r),\Bbb C^{\times})$ is independent of
the choice of each $U(g)$. 

We will construct $\pi_{\psi}$ by means of
Schr\"odinger representations over the finite field $\Bbb F$ in
subsection \ref{subsec:schrodinger-representation} 
(see Proposition
\ref{prop:another-expression-of-pi-beta-psi}), and will show 
in subsection \ref{subsec:action-of-g-o-f-beta} 
that we can construct $U(g)$ by means of Weil representation so that we have 
\begin{equation}
 c_U(g,h)=c_{\beta,\rho}(\overline g,\overline h)
\label{eq:c-u-is-c-beta-theta}
\end{equation}
for all $g,h\in G_{\beta}(O_r)$, where 
$\overline g\in G_{\beta}(\Bbb F)$ is the image of 
$g\in G_{\beta}(O_r)$ under the canonical surjection 
$G(O_r)\to G(\Bbb F)$ (see subsection
\ref{subsec:action-of-g-o-f-beta}), and 
$c_{\beta,\rho}$ is the
Schur multiplier defined in section \ref{sec:schur-multiplier}. 
Furthermore, Proposition 
\ref{prop:vanishing-of-c-beta-rho} tells us that the
Schur multiplier 
$[c_{\beta,\rho}]\in H^2(G_{\beta}(\Bbb F),\Bbb C^{\times})$ is
trivial if the characteristic polynomial of 
$\overline\beta\in M_n(\Bbb F)$ is its minimal polynomial. 
So the Schur multiplier 
$[c_U]\in H^2(G_{\beta}(O_r),\Bbb C^{\times})$ is actually trivial
under the second condition of our main theorem 
\ref{th:main-result}. Now we have 

\begin{prop}
\label{prop:existence-of-canonical-intertwiner-of-pi-psi}
Assume that the Schur multiplier 
$[c_U]\in H^2(G_{\beta}(O_r),\Bbb C^{\times})$ is trivial. Then 
there exists a group homomorphism 
$U_{\psi}:G_{\beta}(O_r)\to GL_{\Bbb C}(V_{\psi})$ such that
\begin{enumerate}
\item $\pi_{\psi}(g^{-1}xg)
       =U_{\psi}(g)^{-1}\circ\pi_{\psi}(x)
                              \circ U_{\psi}(g)$ for all 
      $g\in G_{\beta}(O_r)$ and $x\in K_{l-1}(O_r)$ and
\item $U_{\psi}(h)=1$ for all $h\in G_{\beta}(O_r)\cap K_{l-1}(O_r)$.
\end{enumerate}
\end{prop}
\begin{proof}
Since the Schur multiplier 
$[c_U]\in H^2(G_{\beta}(O_r),\Bbb C^{\times})$ is trivial, there exists a
group homomorphism $U:G_{\beta}(O_r)\to GL_{\Bbb C}(V_{\psi})$ such that 
$\pi_{\psi}(g^{-1}xg)=U(g)^{-1}\circ\pi_{\psi}(x)\circ U(g)$ for all 
$g\in G_{\beta}(O_r)$ and $x\in K_{l-1}(O_r)$. Then for any 
$h\in G_{\beta}(O_r)\cap K_{l-1}(O_r)$ there exists a 
$c(h)\in\Bbb C^{\times}$ such that $U(h)=c(h)\cdot\pi_{\psi}(h)$. 
On the other hand we have 
$$
 G_{\beta}(O_r)\cap K_{l-1}(O_r)\subset Z(O_r,\beta)
$$
since 
$(1+\varpi^{l-1}X)_r\in G_{\beta}(O_r)\cap K_{l-1}(O_r)$
means that
\begin{align*}
 \beta
 &\equiv(1+\varpi^{l-1}X)\beta(1+\varpi^{l-1}X)^{-1}\npmod{\frak{p}^l}\\
 &\equiv(\beta+\varpi^{l-1}X\beta)(1-\varpi^{l-1}X)\npmod{\frak{p}^l}\\
 &\equiv\beta+\varpi^{l-1}[X,\beta]\npmod{\frak{p}^l}
\end{align*}
and then $[X,\beta]\equiv 0\npmod{\frak p}$, that is 
$X\npmod{\frak p}\in\frak{g}_{\beta}(\Bbb F)$. Then 
$\pi_{\psi}(h)$ is the homothety $\psi(h)$ for all 
$h\in G_{\beta}(O_r)\cap K_{l-1}(O_r)$. Extend the group homomorphism 
$h\mapsto c(h)\psi(h)$ of $G_{\beta}(O_r)\cap K_{l-1}(O_r)$ to a group
homomorphism $\theta:G_{\beta}(O_r)\to\Bbb C^{\times}$. Then 
$g\mapsto U_{\psi}(g)=\theta(g)^{-1}U(g)$ is the required group homomorphism.
\end{proof}

Let us denote by 
$G_{\beta}(O_r)\sphat{\times}_{K_{l-1}(O_r)}\frak{g}_{\beta}(\Bbb F)\sphat$
the set of 
$(\theta,\rho)
 \in G_{\beta}(O_r)\sphat\times\frak{g}_{\beta}(\Bbb F)\sphat$ such that 
$\theta=\psi_{\beta,\rho}$ on $G_{\beta}(O_r)\cap K_{l-1}(O_r)$. Then,
 under the assumption of Proposition
 \ref{prop:existence-of-canonical-intertwiner-of-pi-psi}, 
$(\theta,\rho)\in
 G_{\beta}(O_r)\sphat{\times}_{K_{l-1}(O_r)}\frak{g}_{\beta}(\Bbb F)\sphat$ 
defines an irreducible representation $\sigma_{\theta,\rho}$ of 
$G(O_r,\beta)=G_{\beta}(O_r)\cdot K_{l-1}(O_r)$ by
\begin{equation}
 \sigma_{\theta,\rho}(gh)
 =\theta(g)\cdot U_{\psi}(g)\circ\pi_{\psi}(h)
\label{eq:description-of-sigma-beta-theta-for-odd-r}
\end{equation}
for $g\in G_{\beta}(O_r)$ and $h\in K_{l-1}(O_r)$ 
with $\psi=\psi_{\beta,\rho}$. Then we have

\begin{prop}\label{prop:generalized-main-result-for-odd-r}
Assume that the Schur multiplier 
$[c_U]\in H^2(G_{\beta}(O_r),\Bbb C^{\times})$ is trivial. Then 
a bijection of 
$G_{\beta}(O_r)\sphat{\times}_{K_{l-1}(O_r)}\frak{g}_{\beta}(\Bbb F)\sphat$
onto $\text{\rm Irr}(G(O_r,\beta)\mid\psi_{\beta})$ is given by 
$(\theta,\rho)\mapsto\sigma_{\theta,\rho}$.
\end{prop}
\begin{proof}
Clearly 
$\sigma_{\theta,\rho}\in\text{\rm Irr}(G(O_r,\beta)\mid\psi_{\beta})$ for
all 
$(\theta,\rho)\in
 G_{\beta}(O_r)\sphat{\times}_{K_{l-1}(O_r)}\frak{g}_{\beta}(\Bbb F)\sphat$. 
Take a $\sigma\in\text{\rm Irr}(G(O_r,\beta)\mid\psi_{\beta})$. Then 
\begin{align*}
 \sigma\hookrightarrow
 \text{\rm Ind}_{K_l(O_r)}^{G(O_r,\beta)}\psi_{\beta}
 &=\text{\rm Ind}_{Z(O_r,\beta)}^{G(O_r,\beta)}
   \left(\text{\rm Ind}_{K_l(O_r)}^{Z(O_r,\beta)}\psi_{\beta}
         \right)\\
 &=\bigoplus_{\psi\in Y_{\beta}}
   \text{\rm Ind}_{Z(O_r,\beta)}^{G(O_r,\beta)}\psi
\end{align*}
so that there exists a $\psi=\psi_{\beta,\rho}\in Y_{\beta}$ with 
$\rho\in\frak{g}_{\beta}(\Bbb F)\sphat$ such that 
\begin{align*}
 \sigma\hookrightarrow
    \text{\rm Ind}_{Z(O_r,\beta)}^{G(O_r,\beta)}\psi
 &=\text{\rm Ind}_{K_{l-1}(O_r)}^{G(O_r,\beta)}
   \left(\text{\rm Ind}_{Z(O_r,\beta)}^{K_{l-1}(O_r)}\psi
         \right)\\
 &=\bigoplus^{\dim\pi_{\psi}}
   \text{\rm Ind}_{K_{l-1}(O_r)}^{G(O_r,\beta)}\pi_{\psi}
  =\bigoplus^{\dim\pi_{\psi}}\bigoplus_{\theta}\sigma_{\theta,\psi},
\end{align*}
where $\bigoplus_{\theta}$ is the direct sum over 
$\theta\in G_{\beta}(O_r)\sphat$ such that $\theta=\psi_{\beta,\rho}$ on 
$G_{\beta}(O_r)\cap K_{l-1}(O_r)$. 
Then we have $\sigma=\sigma_{\theta,\rho}$ for some 
$(\theta,\rho)\in
 G_{\beta}(O_r)\sphat{\times}_{K_{l-1}(O_r)}\frak{g}_{\beta}(\Bbb F)\sphat$.
\end{proof}

We have also

\begin{prop}\label{prop:fundamental-bijection-of-parameter}
$(\theta,\rho)\mapsto\theta$ gives a bijection of 
$G_{\beta}(O_r)\sphat{\times}_{K_{l-1}(O_r)}\frak{g}_{\beta}(\Bbb F)\sphat$ 
onto the set
$$
  \left\{\theta\in G_{\beta}(O_r)\sphat\;\;\;
         \text{\rm s.t. $\theta=\psi_{\beta}$ on 
               $G_{\beta}(O_r)\cap K_l(O_r)$}\right\}.
$$
\end{prop}
\begin{proof}
Take a 
$(\theta,\rho)\in 
 G_{\beta}(O_r)\sphat{\times}_{K_{l-1}(O_r)}\frak{g}_{\beta}(\Bbb F)\sphat$.
The smoothness of $G_{\beta}$ over $O$ implies that the canonical
mapping $\frak{g}_{\beta}(O)\to\frak{g}_{\beta}(\Bbb F)$ is
surjective. So Take a $\overline X\in\frak{g}_{\beta}(\Bbb F)$ with 
$X\in\frak{g}_{\beta}(O)$. Then we have
$$
 g=1+\varpi^{l-1}X+2^{-1}\varpi^{2l-2}X^2\npmod{\frak{p}^r}
 \in K_{l-1}(O_r)\cap G_{\beta}(O_r)
$$
so that
\begin{align*}
 \theta(g)=\psi_{\beta,\rho}(g)
 &=\tau\left(\varpi^{-l}B(X+2^{-1}\varpi^{l-1}X^2,\beta)
             -2^{-1}\varpi^{-1}B(X,\beta)\right)\cdot
   \rho(\overline X)\\
 &=\tau\left(\varpi^{-l}B(X,\beta)\right)\cdot\rho(\overline X).
\end{align*}
Hence we have 
$$
 \rho(\overline X)
 =\tau\left(-\varpi^{-l}B(X,\beta)\right)\cdot
  \theta\left(1+\varpi^{l-1}X+2^{-1}\varpi^{2l-2}X^2\npmod{\frak{p}^r}
              \right).
$$
This means that the mapping $(\theta,\rho)\mapsto\theta$ is
injective. Take $X, X^{\prime}\in\frak{g}_{\beta}(O)$ such that 
$X\equiv X^{\prime}\pmod{\frak p}$. Then we have 
$X^{\prime}=X+\varpi T$ with $T\in\frak{g}_{\beta}(O)$ and
\begin{align*}
 &1+\varpi^{l-1}X^{\prime}
                   +2^{-1}\varpi^{2l-2}X^{\prime 2}\pmod{\frak{p}^r}\\
 =&1+\varpi^{-1}X+2^{-1}\varpi^{2l-2}X^2+\varpi^lT\pmod{\frak{p}^r}\\
 =&(1+\varpi^{l-1}X+2^{-1}\varpi^{2l-2}X^2)(1+\varpi^lT)
         \pmod{\frak{p}^r},
\end{align*}
where $1+\varpi^lT\pmod{\frak{p}^r}\in K_l(O_r)$ and hence
$$
 \theta(1+\varpi^lT\npmod{\frak{p}^r})
 =\psi_{\beta}(1+\varpi^lT\npmod{\frak{p}^r})
 =\tau\left(\varpi^{-(l-1)}B(T,\beta)\right).
$$
This and the commutativity of $G_{\beta}$ show that
$$
 \rho(\overline X)
 =\tau\left(-\varpi^{-l}B(X,\beta)\right)\cdot
  \theta\left(1+\varpi^{l-1}X+2^{-1}\varpi^{2l-2}X^2\npmod{\frak{p}^r}
              \right)
$$
with $\overline X\in\frak{g}_{\beta}(\Bbb F)$ with
$X\in\frak{g}_{\beta}(O)$ gives an well-defined group homomorphism of 
$\frak{g}_{\beta}(\Bbb F)$ to $\Bbb C^{\times}$. Then 
$(\theta,\rho)\in 
 G_{\beta}(O_r)\sphat{\times}_{K_{l-1}(O_r)}\frak{g}_{\beta}(\Bbb F)\sphat$ 
and our mapping in question is surjective.
\end{proof}
         
Proposition \ref{prop:generalized-main-result-for-odd-r} and 
Proposition \ref{prop:fundamental-bijection-of-parameter} give the
bijection presented in our main Theorem 
\ref{th:main-result} in the case of $r$ being odd. So all that we need
is to show the triviality of the Schur multiplier 
$[c_U]\in H^2(G_{\beta}(O_r),\Bbb C^{\times})$ in 
Proposition \ref{prop:generalized-main-result-for-odd-r}. The
triviality is proved in subsection
\ref{subsec:vanishing-of-c-beta-rho} after 
\begin{enumerate}
\item analyzing, in subsections
      \ref{subsec:k-l-1-as-group-extension} and  
      \ref{subsec:reduction-to-intermediate-group}, 
      the structure of $K_{l-1}(O_r)$ and 
\item describing, in subsection
      \ref{subsec:schrodinger-representation} and 
      \ref{subsec:action-of-g-o-f-beta},  the Schur multiplier 
$$
 [c_U]\in H^2(G_{\beta}(O_r),\Bbb C^{\times})
$$
      in terms of 
      the Schur multiplier defined in section 
      \ref{sec:schur-multiplier}.
\end{enumerate}

\subsection[]{Structure of $K_{l-1}(O_r)$}
\label{subsec:k-l-1-as-group-extension}
A group extension
\begin{equation}
 0\to\frak{g}(O_{l-1})\xrightarrow{\diamondsuit}
     K_{l-1}(O_r)\xrightarrow{\heartsuit}
     \frak{g}(\Bbb F)\to 0
\label{eq:k-l-1-as-group-extension}
\end{equation}
is given by the canonical surjection 
\eqref{eq:canonical-surjection-of-k-l-1-to-g(f)}, whose kernel is
$K_l(O_r)$, with the group isomorphism 
$$
 \diamondsuit:\frak{g}(O_{l-1})\,\tilde{\to}\,K_l(O_r)
$$
defined by 
$S\npmod{\frak{p}^{l-1}}\mapsto(1+\varpi^lS)\npmod{\frak{p}^r}$ which
is assumed as the condition II) in subsection
\ref{subsec:fundamental-setting}. 

In order to determine the 2-cocycle of the group extension 
\eqref{eq:k-l-1-as-group-extension}, choose any mapping 
$\lambda:\frak{g}(\Bbb F)\to\frak{g}(O)$ such that 
$X=\lambda(X)\npmod{\frak p}$ for all $X\in\frak{g}(\Bbb F)$ and 
$\lambda(0)=0$, and define a section
$$
 l:\frak{g}(\Bbb F)\to K_{l-1}(O_r)
$$
of \eqref{eq:k-l-1-as-group-extension} by 
$X\mapsto 1+\varpi^{l-1}\lambda(X)
           +2^{-1}\varpi^{2l-2}\lambda(X)^2\npmod{\frak{p}^r}$. 
Then we have
$$
 l(X)^{-1}
 =1-\varpi^{l-1}\lambda(X)+2^{-1}\varpi^{2l-2}\lambda(X)^2
  \npmod{\frak{p}^r}
$$
for all $X\in\frak{g}(\Bbb F)$ and
$$
 l(X)(1+\varpi^lS)l(X)^{-1}
 \equiv 1+\varpi^lS\npmod{\frak{p}^r}
$$
for all $S_{l-1}\in\frak{g}(O_{l-1})$. Furthermore we have
$$
 l(X)l(Y)l(X+Y)^{-1}
 =1+\varpi^l\left\{\mu(X,Y)+2^{-1}\varpi^{l-2}[\lambda(X),\lambda(Y)]
                  \right\}\npmod{\frak{p}^r}
$$
for all $X,Y\in\frak{g}(\Bbb F)$ where 
$\mu:\frak{g}(\Bbb F)\times\frak{g}(\Bbb F)\to\frak{g}(O)$ 
is defined by 
$$
 \lambda(X)+\lambda(Y)-\lambda(X+Y)
 =\varpi\cdot\mu(X,Y)
$$
for all $X,Y\in\frak{g}(\Bbb F)$. Now we have two elements 
($2$-cocycle) 
$$
 \mu=[(X,Y)\mapsto\mu(X,Y)_{l-1}],
 \quad
 c=[(\overline X,\overline Y)\mapsto 2^{-1}\varpi^{l-2}[X,Y]_{l-1}]
$$
of $Z^2(\frak{g}(\Bbb F),\frak{g}(O_{l-1}))$ 
with trivial action of $\frak{g}(\Bbb F)$ on $\frak{g}(O_{l-1})$. 

Let us consider two groups $\Bbb M$ and $\Bbb G$ corresponding to the two
2-cocycles $\mu$ and $c$ respectively. That is the group operation on 
$\Bbb M=\frak{g}(\Bbb F)\times\frak{g}(O_{l-1})$ is defined by
$$
 (X,S_{l-1})\cdot(Y,T_{l-1})=(X+Y,(S+T+\mu(X,Y))_{l-1})
$$
and the group operation on 
$\Bbb G=\frak{g}(\Bbb F)\times\frak{g}(O_{l-1})$ is defined by 
$$
 (\overline X,S_{l-1})\cdot(\overline Y,T_{l-1})
 =(\overline{X+Y},(S+T+2^{-1}\varpi^{l-2}[X,Y])_{l-1}).
$$
Let $\Bbb G{\times}_{\frak{g}(\Bbb F)}\Bbb M$ be the fiber product of 
$\Bbb G$ and $\Bbb M$ with respect to the canonical projections 
onto $\frak{g}(\Bbb F)$. In other words
$$
 \Bbb G{\times}_{\frak{g}(\Bbb F)}\Bbb M
 =\left\{(X;S,T)=((X,S),(X,T))\in\Bbb G\times\Bbb M\right\}
$$
is a subgroup of the direct product $\Bbb G\times\Bbb M$. 
We have a surjective group homomorphism 
\begin{equation}
 (\ast):\Bbb G{\times}_{\frak{g}(\Bbb F)}\Bbb M\to K_{l-1}(O_r)
\label{eq:fundamental-surjection-for-general-alg-group}
\end{equation}
defined by 
\begin{align*}
 (X;S_{l-1},T_{l-1})\mapsto&
                      l(X)\cdot(1+\varpi^l(S+T)\npmod{\frak{p}^r})\\
         &=1+\varpi^{l-1}\lambda(X)
            +2^{-1}\varpi^{2l-2}\lambda(X)^2
            +\varpi^l(S+T)\npmod{\frak{p}^r}.
\end{align*}

\subsection[]{Reduction to an intermediate group}
\label{subsec:reduction-to-intermediate-group}
The group homomorphism 
$B_{\beta}:\frak{g}(O_{l-1})\to O_{l-1}$ ($X\mapsto B(X,\beta_{l-1})$)
induces a group homomorphism 
$$
 B_{\beta}^{\ast}:H^2(\frak{g}(\Bbb F),\frak{g}(O_{l-1}))\to 
                  H^2(\frak{g}(\Bbb F),O_{l-1}).
$$
Let us denote by 
$\mathcal{H}_{\beta}$ the group associated with the 2-cocycle 
$$
 c_{\beta}=B_{\beta}\circ c
 =\left[(\overline X,\overline Y)\mapsto
        2^{-1}\varpi^{l-2}B([X,Y],\beta)_{l-1}\right]
 \in Z^2(\frak{g}(\Bbb F),O_{l-1}).
$$
That is $\mathcal{H}_{\beta}=\frak{g}(\Bbb F)\times O_{l-1}$ with a
group operation 
$$
 (\overline X,s)\cdot(\overline Y,t)
 =(\overline{X+Y},s+t+2^{-1}\varpi^{l-2}B([X,Y],\beta)_{l-1}).
$$
Then the center of $\mathcal{H}_{\beta}$ is 
$Z(\mathcal{H}_{\beta})=\frak{g}_{\beta}(\Bbb F)\times O_{l-1}$, the
direct product of two additive groups $\frak{g}_{\beta}(\Bbb F)$ and
$O_{l-1}$. 

The inverse
image of $Z(\mathcal{H}_{\beta})$ with respect to the surjective group
homomorphism 
\begin{equation}
 \spadesuit:\Bbb G{\times}_{\frak{g}(\Bbb F)}\Bbb M\to\mathcal{H}_{\beta}
 \qquad
 \left((X;S_{l-1},T_{l-1})\mapsto
       (X,B(S,\beta)_{l-1})\right)
\label{eq:fund-surjection-on-hesenberg-like-group-for-general-alg-group}
\end{equation}
is 
$\left(\Bbb G{\times}_{\frak{g}(\Bbb F)}\Bbb M\right)_{\beta}
 =\left\{(X;S,T)\in\Bbb G{\times}_{\frak{g}(\Bbb F)}\Bbb M\mid
         X\in\frak{g}(\Bbb F)_{\overline\beta}\right\}$ 
which is 
mapped onto $Z(O_r,\beta)\subset K_{l-1}(O_r)$ by the surjection 
\eqref{eq:fundamental-surjection-for-general-alg-group}. 

Take a $\rho\in\frak{g}_{\beta}(\Bbb F)\sphat$ which defines group
homomorphisms
$$
 \chi_{\rho}
 =\rho\otimes\left[x_{l-1}\mapsto\tau(\varpi^{-(l-1)}x)\right]
 :Z(\mathcal{H}_{\beta})
  =\frak{g}_{\beta}(\Bbb F)\times O_{l-1}
 \to\Bbb C^{\times}
$$
and
$$
 \widetilde{\chi}_{\rho}:
  \left(\Bbb G{\times}_{\frak{g}(\Bbb F)}\Bbb M\right)_{\beta}
  \xrightarrow{\spadesuit}Z(\mathcal{H}_{\beta})
  \xrightarrow{\chi_{\rho}}\Bbb C^{\times}.
$$
On the other hand we have a group homomorphism
$$
 \widetilde{\psi}_0:
 \Bbb G{\times}_{\frak{g}(\Bbb F)}\Bbb M\to\Bbb C^{\times}
$$
defined by 
$\widetilde{\psi}_0(X;S_{l-1},T_{l-1})
 =\tau\left(\varpi^{-l}B(\lambda(X)+\varpi T,\beta)\right)$. Then 
$\widetilde{\psi}_0\cdot\widetilde{\chi}_{\beta}$ is trivial on the
 kernel of the surjection
 \eqref{eq:fundamental-surjection-for-general-alg-group}
 and it induces a group homomorphism $\psi_{\beta,\rho}\in Y_{\beta}$
 defined in subsection
 \ref{subsec:crude-weil-representation}. 

\subsection[]{Schr\"odinger representations over finite fields}
\label{subsec:schrodinger-representation}
Fix a $\rho\in\frak{g}_{\beta}(\Bbb F)\sphat$. 
Let us determine the 2-cocycle of the group extension
\begin{equation}
 0\to Z(\mathcal{H}_{\beta})\to\mathcal{H}_{\beta}
  \xrightarrow{\clubsuit}\Bbb V_{\beta}\to 0
\label{eq:central-extension-of-h-beta}
\end{equation}
where $\clubsuit:\mathcal{H}_{\beta}\to\Bbb V_{\beta}$
is defined by 
$(X,s)\mapsto\dot X\npmod{\frak{g}_{\beta}(\Bbb F)}$. 
Fix a 
$\Bbb F$-linear section $v\mapsto[v]$ of the exact sequence
$$
 0\to\frak{g}_{\beta}(\Bbb F)\to\frak{g}(\Bbb F)
  \to\Bbb V_{\beta}\to 0
$$
of $\Bbb F$-vector spaces and define a section 
$l:\Bbb V_{\beta}\to\mathcal{H}_{\beta}$ of the group
  extension \eqref{eq:central-extension-of-h-beta} by 
$l(v)=([v],0)$. Then we have
$$
 l(u)l(v)l(u+v)^{-1}
 =(0,2^{-1}\varpi^{l-2}B([X,Y],\beta)\npmod{\frak{p}^{l-1}})
$$
for 
$u=\dot{\overline X}, v=\dot{\overline Y}
 \in\Bbb V_{\beta}$ so that the 2-cocycle of the group extension 
\eqref{eq:central-extension-of-h-beta} is
$$
 \left[(\dot{\overline X},\dot{\overline Y})\mapsto
       2^{-1}\varpi^{l-2}B([X,Y],\beta)\npmod{\frak{p}^{l-1}}\right]
 \in Z^2(\Bbb V_{\beta},O_{l-1}).
$$
Define a group operation on 
$\Bbb H_{\beta}=\Bbb V_{\beta}\times Z(\mathcal{H}_{\beta})$
by 
$$
 \left(\dot{\overline X},z\right)\cdot\left(\dot{\overline Y},w\right)
 =\left(\dot{\overline{X+Y}},
    z+w+2^{-1}\varpi^{l-2}B([X,Y],\beta)\npmod{\frak{p}^{l-1}}\right).
$$
Then $\Bbb H_{\beta}$ is isomorphic to $\mathcal{H}_{\beta}$ by 
$(v,(Y,s))\mapsto([v]+Y,s)$.

 Let $H_{\beta}$ be the Heisenberg group
of the symplectic $\Bbb F$-space $\Bbb V_{\beta}$, that is 
$H_{\beta}=\Bbb V_{\beta}\times\Bbb C^1$ with a group
operation
$$
 (u,s)\cdot(v,t)
 =\left(u+v,s\cdot t\cdot
            \widehat\tau(2^{-1}\langle u,v\rangle_{\beta})
        \right).
$$
Then we have a group homomorphism 
$$
 \Bbb H_{\beta}=\Bbb V_{\beta}\times Z(\mathcal{H}_{\beta})\to H_{\beta}
 \qquad
 ((v,z)\mapsto(v,\chi_{\rho}(z)).
$$
Fix a polarization 
$\Bbb V_{\beta}=\Bbb W^{\prime}\oplus\Bbb W$ of the symplectic
$\Bbb F$-space $\Bbb V_{\beta}$. Let us denote by 
$L^2(\Bbb W^{\prime})$ 
the complex vector space of the complex-valued functions $f$ on
$\Bbb W^{\prime}$ with inner product 
$(f,f^{\prime})
 =\sum_{w\in\Bbb W^{\prime}}f(w)\overline{f^{\prime}(w)}$.
The Schr\"odinger representation $(\pi^{\beta},L^2(\Bbb W^{\prime}))$ 
of $H_{\beta}$ associated with the polarization is defined for 
$(v,s)\in H_{\beta}$ and $f\in L^2(\Bbb W^{\prime})$ by
$$
 \left(\pi^{\beta}(v,s)f\right)(w)
 =s\cdot\widehat\tau\left(
   2^{-1}\langle v_-,v_+\rangle_{\overline\beta}
   +\langle w,v_+\rangle_{\overline\beta}\right)\cdot f(w+v_-)
$$
where $v=v_-+v_+\in\Bbb V_{\beta}$ with 
$v_-\in\Bbb W^{\prime}, v_+\in\Bbb W^{\prime}$. 

Now an irreducible
representation $(\pi^{\beta,\rho},L^2(\Bbb W^{\prime}))$ of 
$\Bbb H_{\beta}$ is defined by 
$\pi^{\beta,\rho}(v,z)=\pi^{\beta}(v,\chi_{\rho}(z))$, and an
irreducible representation 
$(\widetilde\pi^{\beta,\rho},L^2(\Bbb W^{\prime}))$ of 
$\Bbb G{\times}_{\frak{g}(\Bbb F)}\Bbb M$ is defined by
$$
 \widetilde\pi^{\beta,\rho}:
 \Bbb G{\times}_{\frak{g}(\Bbb F)}\Bbb M
 \xrightarrow{\spadesuit}\mathcal{H}_{\beta}
 \,\tilde{\to}\,\Bbb H_{\beta}
 \xrightarrow{\pi^{\beta,\rho}}
 GL_{\Bbb C}(L^2(\Bbb W^{\prime})).
$$
Then $\widetilde{\psi}_0\cdot\widetilde{\pi}_{\beta,\rho}$ is trivial
on the kernel of 
$(\ast):\Bbb G{\times}_{\frak{g}(\Bbb F)}\Bbb M\to K_{l-1}(O_r)$ so
that it induces an irreducible representation 
$\pi_{\beta,\rho}$ of $K_{l-1}(O_r)$ on $L^2(\Bbb W^{\prime})$. 

\begin{prop}
\label{prop:another-expression-of-pi-beta-psi}
Take a $g=1+\varpi^{l-1}T\npmod{\frak{p}^r}\in K_{l-1}(O_r)$ 
with $T\in\frak{gl}_n(O)$. Then we have 
$T\npmod{\frak{p}^{l-1}}\in\frak{g}(O_{l-1})$ and
$$
 \pi_{\beta,\rho}(g)
  =\tau\left(\varpi^{-l}B(T,\beta)
            -2^{-1}\varpi^{-1}B(T^2,\beta)\right)\cdot
   \rho(Y)\cdot\pi^{\beta}(v,1)
$$
where $\overline T=[v]+Y\in\frak{g}(\Bbb F)$ with 
$v\in\Bbb V_{\beta}$ and 
$Y\in\frak{g}_{\beta}(\Bbb F)$. In particular
$\pi_{\beta,\rho}(h)$ is the homothety $\psi_{\beta,\rho}(h)$ for all 
$h\in Z(O_r,\beta)$.
\end{prop}
\begin{proof}
By the definition we have
\begin{align*}
 \pi_{\beta,\psi}(&l(X)(1+\varpi^lS\npmod{\frak{p}^r}))\\
 =&\psi_0(X,0)\cdot\widetilde\pi^{\beta,\rho}([v]+Y;S_{l-1},0)\\
 =&\psi\left(l(Y)(1+\varpi^lS\npmod{\frak{p}^r})\right)\cdot
   \tau\left(\varpi^{-l}B(\lambda(X)-\lambda(Y),\beta)\right)\cdot
   \pi^{\beta}(v,1)\\
 =&\tau\left(
     \varpi^{-(l-1)}B(S,\beta)+\varpi^{-l}B(\lambda(X),\beta)\right)
      \cdot\rho(Y)\cdot\pi^{\beta}(v,1)
\end{align*}
where $X=[v]+Y\in\frak{g}(\Bbb F)$ with $v\in\Bbb V_{\beta}$
and $Y\in\frak{g}_{\beta}(\Bbb F)$ and 
put $1+\varpi^{l-1}T\equiv l(X)(1+\varpi^lS)\npmod{\frak{p}^r}$ with 
$X\in\frak{g}(\Bbb F)$ and $S\in\frak{g}(O)$. Then we have
$$
 1+\varpi^{l-1}T\equiv 1+\varpi^{l-1}\lambda(X)
 \npmod{\frak{p}^l}
$$
so that we have $T\npmod{\frak p}=X\in\frak{g}(\Bbb F)$ and 
$$
 \varpi S\equiv 
 T-\lambda(\overline T)-2^{-1}\varpi^{l-1}\lambda(\overline T)^2
 \npmod{\frak{p}^l}.
$$
Then we have
\begin{align*}
 \pi_{\beta,\rho}(g)
 =&\pi^{\beta,\rho}\left(l(\overline T)(1+\varpi^lS)_r\right)\\
 =&\tau\left(\varpi^{-(l-1)}B(S,\beta)
             +\varpi^{-l}B(\lambda(\overline T),\beta)\right)\cdot
   \rho(Y)\cdot\pi^{\beta}(v,1)\\
 =&\tau\left(\varpi^{-l}B(T,\beta)-2^{-1}\varpi^{-1}B(T^2,\beta)
             \right)\cdot\rho(Y)\cdot\pi^{\beta}(v,1).
\end{align*}
\end{proof}

This proposition shows that the irreducible representation 
$(\pi_{\beta,\rho},L^2(\Bbb W^{\prime}))$ $K_{l-1}(O_r)$ 
is exactly the irreducible representation $\pi_{\psi}$ with 
$\psi=\psi_{\beta,\rho}\in Y_{\beta}$ defined in Proposition
\ref{prop:existence-of-fundamental-rep-associated-with-psi}. 

\subsection[]{Description of Schur multiplier}
\label{subsec:action-of-g-o-f-beta}
Fix a $\rho\in\frak{g}_{\beta}(\Bbb F)\sphat$. 
In this subsection we will study the conjugate action of 
$g_r=g\npmod{\frak{p}^r}\in G(O_r,\beta)$ on $K_{l-1}(O_r)$ and on 
$\pi_{\beta,\rho}$. For any $X\in\frak{g}(\Bbb F)$, we have
$$
 g_r^{-1}l(X)g_r
 =l\left(\text{\rm Ad}(\overline g)^{-1}X)
         +\varpi^l\nu(X,g)\npmod{\frak{p}^r}\right)
$$
with $\nu(X,g)\in\frak{g}(O)$ such that 
$$
 \text{\rm Ad}(g)^{-1}\lambda(X)-\lambda(\text{\rm Ad}(g)^{-1}X)
 =\varpi\cdot\nu(X,g).
$$
Then we have
\begin{align*}
 &g_r^{-1}l(X)(1+\varpi^l(S+T)_r)g_r\\
 =&l(\text{\rm Ad}(\overline g)^{-1}X)
   \left(1+\varpi^l(\text{\rm Ad}(g)^{-1}S+\text{\rm Ad}(g)^{-1}T
          +\nu(X,g))\right)\npmod{\frak{p}^r}
\end{align*}
and an action of $g_r\in G(O_r,\beta)$ on 
$(X;S_{l-1},T_{l-1})
 \in\Bbb G{\times}_{\frak{g}(\Bbb F)}\Bbb M$ is defined by 
\begin{equation}
 (X;S_{l-1},T_{l-1})^{g_r}
 =\left(\text{\rm Ad}(\overline g)^{-1}X;
         (\text{\rm Ad}(g)^{-1}S)_{l-1},
         (\text{\rm Ad}(g)^{-1}T+\nu(X,g))_{l-1}\right)
\label{eq:action-on-fibre-product}.
\end{equation}
The action \eqref{eq:action-on-fibre-product} is compatible with the
action  
$$
 (X,s)^{g_r}=(\text{\rm Ad}(\overline g)^{-1}X,s)
$$ 
of $g_r\in G(O_r,\beta)$ on $(X,s)\in\mathcal{H}_{\beta}$ via
the surjection
\eqref{eq:fund-surjection-on-hesenberg-like-group-for-general-alg-group}. 
If
we put $X=[v]+Y\in\frak{g}(\Bbb F)$ with $v\in\Bbb V_{\beta}$
and $Y\in\frak{g}_{\beta}(\Bbb F)$, then we have
$$
 \text{\rm Ad}(\overline g)^{-1}X
 =[v\sigma_{\overline g}]+\gamma(v,\overline g)
                        +\text{\rm Ad}(\overline g)^{-1}Y
$$
in the notations of subsection
\ref{subsec:definition-of-schur-multiplier}. So 
$g_r\in G(O_r,\beta)$ acts on $(v,(Y,s))\in\Bbb H_{\beta}$ by 
$$
 (v,(Y,s))^{g_r}
 =\left(v\sigma_{\overline g},
   (\text{\rm Ad}(\overline g)^{-1}Y+\gamma(v,\overline g),s)\right).
$$
In particular $g_r\in G_{\beta}(O_r)$ acts on 
$(v,z)\in\Bbb H_{\beta}$ by 
$$
 (v,z)^{g_r}
 =(v\sigma_{\overline g},(\gamma(v,\overline g),0)\cdot z).
$$
There exists a group homomorphism 
$T:Sp(\Bbb V_{\beta})\to GL_{\Bbb C}(L^2(\Bbb W^{\prime}))$
  such that 
$$
 \pi^{\beta}(v\sigma,s)
 =T(\sigma)^{-1}\circ\pi^{\beta}(v,s)\circ T(\sigma)
$$
for all $\sigma\in Sp(\Bbb V_{\beta})$ and $(v,s)\in
H_{\beta}$ (see \cite[Th.2.4]{Gerardin1977}). 
Then we have
\begin{align*}
 \pi^{\beta,\rho}((v,z)^{g_r})
 =&\pi^{\beta,\rho}
    \left(v\sigma_{\overline g},(\gamma(v,\overline g),0)\cdot z\right)\\
 =&\pi^{\beta}
    \left(v\sigma_{\overline g},
           \rho(\gamma(v,\overline g))\cdot\chi_{\rho}(z)\right)\\
 =&\pi^{\beta}(v\sigma_{\overline g},
    \widehat\tau\left(\langle v,v_{\overline g}\rangle_{\beta}
                       \right)\cdot\chi_{\rho}(z))\\
 =&T(\sigma_{\overline g})^{-1}\circ
   \pi^{\beta}(v,
    \widehat\tau\left(\langle v,v_{\overline g}\rangle_{\beta}
                       \right)\cdot\chi_{\rho}(z))\circ
   T(\sigma_{\overline g})\\
 =&T(\sigma_{\overline g})^{-1}\circ
   \pi^{\beta}
    \left((v_{\overline g},1)^{-1}(v,\chi_{\rho}(z))(v_{\overline g},1)
           \right)\circ
   T(\sigma_{\overline g})\\
 =&T(\sigma_{\overline g})^{-1}\circ
   \pi^{\beta}(v_{\overline g},1)^{-1}\circ
   \pi^{\beta,\rho}(v,z)\circ
   \pi^{\beta}(v_{\overline g},1)\circ
   T(\sigma_{\overline g}).
\end{align*}
If we put
$$
 U(g_r)
 =\pi^{\beta}(v_{\overline g},1)\circ T(\sigma_{\overline g})
 \in GL_{\Bbb C}(L^2(\Bbb W^{\prime}))
$$
for $g_r\in G_{\beta}(O_r)$ then we have
$$
 U(g_r)\circ U(h_r)
 =c_{\beta,\rho}(\overline g,\overline h)\cdot U((gh)_r)
$$
for all $g_r,h_r\in G_{\beta}(O_r)$, in fact 
\begin{align*}
 U(g_r)\circ U(h_r)
 &=\pi^{\beta}(v_{\overline g},1)\circ T(\sigma_{\overline g})\circ 
   \pi^{\beta}(v_{\overline h},1)\circ T(\sigma_{\overline h})\\
 &=\pi^{\beta}(v_{\overline g},1)\circ 
   \pi^{\beta}(v_{\overline h}^{{\overline g}^{-1}},1)\circ
   T(\sigma_{\overline g})\circ T(\sigma_{\overline h})\\
 &=\pi^{\beta}\left(v_{\overline g}+v_{{\overline h}^{{\overline g}^{-1}}},
    \widehat\tau\left(
      2^{-1}\langle v_{\overline g},v_{{\overline h}^{{\overline g}^{-1}}}\rangle_{\beta}
                    \right)\right)\circ
   T(\sigma_{\overline{gh}})\\
 &=c_{\beta,\rho}(\overline g,\overline h)\cdot
   \pi^{\beta}(v_{\overline{gh}},1)\circ T(\sigma_{\overline{gh}}).
\end{align*}
On the other hand 
\begin{align*}
  \widetilde\psi_0\left((X;S_{l-1},T_{l-1})^{g_r^{-1}}
                       \right)
 &=\tau\left(\varpi^{-l}B(\lambda(X)+\varpi T,\text{\rm Ad}(g)\beta)
             \right)\\
 &=\tau\left(\varpi^{-l}B(\lambda(X)+\varpi T,\beta)
             \right)
\end{align*}
for all $g_r\in G_{\beta}(O_r)$. That is $\widetilde{\psi}_0$
is invariant under the conjugate action of $G_l(O_r,\beta)^{(c)}$. 
Hence we have
$$
 \pi_{\beta,\psi}(g_r^{-1}hg_r)
 =U(g_r)^{-1}\circ\pi_{\beta,\psi}(h)\circ
  U(g_r)
$$
for all $g_r\in G_{\beta}(O_r)$ and $h\in K_{l-1}(O_r)$. 

\subsection[]{Triviality of Schur multiplier}
\label{subsec:vanishing-of-c-beta-rho}
The following proposition is the keystone of this paper.

\begin{prop}\label{prop:vanishing-of-c-beta-rho}
If the characteristic polynomial of 
$\overline\beta\in\frak{g}(\Bbb F)\subset\frak{gl}_n(\Bbb F)$ is the
minimal polynomial of $\overline\beta\in M_n(\Bbb F)$, then the Schur
  multiplier 
$[c_{\beta,\rho}]\in H^2(G_{\beta}(\Bbb F),\Bbb C^{\times})$ is
  trivial for all $\rho\in\frak{g}(\Bbb F)\sphat$.
\end{prop}
\begin{proof}
We will divide the proof into two parts.

1) The case of $G=GL_n$. In this case, Corollary 5.1 of 
\cite{Stasinski-Stevens2017} shows that the Schur multiplier 
$[c_U]\in H^2(G_{\beta}(O_r),\Bbb C^{\times})$ is trivial. On the
other hand we have the inflation-restriction exact sequence 
\begin{align*}
 1&\to H^1(G_{\beta}(\Bbb F),\Bbb C^{\times})
   \xrightarrow{\text{\rm inf}}
       H^1(G_{\beta}(O_r),\Bbb C^{\times})
   \xrightarrow{\text{\rm res}}
       H^1(K_1(O_r)\cap G_{\beta}(O_r),\Bbb C^{\times})^{G_{\beta}(\Bbb F)}\\
  &\to H^2(G_{\beta}(\Bbb F),\Bbb C^{\times})
   \xrightarrow{\text{\rm inf}}
       H^2(G_{\beta}(O_r),\Bbb C^{\times})
\end{align*}
induced by the exact sequence
$$
 1\to K_1(O_r)\cap G_{\beta}(O_r)\to G_{\beta}(O_r)\to G_{\beta}(\Bbb F)\to 1.
$$
Since we have
\begin{align*}
 H^1(G_{\beta}(O_r),\Bbb C^{\times})
 &=\text{\rm Hom}(G_{\beta}(O_r),\Bbb C^{\times}),\\
 H^1(K_1(O_r)\cap G_{\beta}(O_r),\Bbb C^{\times})^{G_{\beta}(\Bbb F)}
 &=\text{\rm Hom}(K_1(O_r)\cap G_{\beta}(O_r),\Bbb C^{\times})
\end{align*}
and $K_1(O_r)\cap G_{\beta}(O_r)\subset G_{\beta}(O_r)$ are finite
commutative groups, the restriction mapping
$$
 \text{\rm res}:H^1(G_{\beta}(O_r),\Bbb C^{\times})\to
  H^1(K_1(O_r)\cap G_{\beta}(O_r),\Bbb C^{\times})^{G_{\beta}(\Bbb F)}
$$
is surjective. Hence the inflation mapping
$$
 \text{\rm inf}:H^2(G_{\beta}(\Bbb F),\Bbb C^{\times})\to
                H^2(G_{\beta}(O_r),\Bbb C^{\times})
$$
is injective. Since the results of the preceding subsections show that 
the Schur multiplier 
$[c_U]\in H^2(G_{\beta}(O_r),\Bbb C^{\times})$ is the image of 
$[c_{\beta,\rho}]\in H^2(G_{\beta}(\Bbb F),\Bbb C^{\times})$ under the
inflation mapping, the statement of the proposition is established
for the group $G=GL_n$.
\renewcommand{\thefootnote}{\fnsymbol{footnote}}
\footnote[2]{This argument is presented by the referee.}

2) The general case of $G\subset GL_n$. We have 
$G_{\beta}(\Bbb F)\subset GL_{n,\beta}(\Bbb F)$. Then Proposition 
\ref{prop:regular-schur-multiplier-is-restriction-from-upper-group}
says that the Schur multiplier 
$[c_{\beta,\rho}]\in H^2(G_{\beta}(\Bbb F),\Bbb C^{\times})$ is the
image of the Schur multiplier 
$[c_{\beta,\widetilde\rho}]\in 
 H^2(GL_{n,\beta}(\Bbb F),\Bbb C^{\times})$ under the restriction
 mapping
$$
 \text{\rm res}:H^2(GL_{n,\beta}(\Bbb F),\Bbb C^{\times})\to
                H^2(G_{\beta}(\Bbb F),\Bbb C^{\times}).
$$
Since we have shown in the part one of the proof that 
$[c_{\beta,\widetilde\rho}]\in H^2(GL_{n,\beta},\Bbb C^{\times})$ is
trivial, so is 
$[c_{\beta,\rho}]\in H^2(G_{\beta}(\Bbb F),\Bbb C^{\times})$.
\end{proof}

Now we have established the triviality of the Schur multiplier 
$$
 [c_{\beta,\rho}]\in H^2(G_{\beta}(\Bbb F),\Bbb C^{\times})
$$ 
which implies the triviality of the Schur multiplier 
$[c_U]\in H^2(G_{\beta}(O_r),\Bbb C^{\times})$ due to the relation 
\eqref{eq:c-u-is-c-beta-theta}. Then 
Proposition \ref{prop:generalized-main-result-for-odd-r} and 
Proposition \ref{prop:fundamental-bijection-of-parameter} give the
bijection presented in our main Theorem 
\ref{th:main-result} in the case of $r$ being odd.

It may be quite interesting if we can find a counter example to the
following statement;
\begin{quote}
Let $G$ be a connected reductive algebraic group defined over 
$\Bbb F$ and $\frak{g}$ the Lie algebra scheme of $G$. Take a 
$\beta\in\frak{g}(\Bbb F)$ which is smoothly regular with respect to $G$ and
  $G_{\beta}$ is commutative. Then the Schur multiplier 
$[c_{\beta,\rho}]\in H^2(G_{\beta}(\Bbb F),\Bbb C^{\times})$ is trivial
  for all $\rho\in\frak{g}(\Bbb F)\sphat$.
\end{quote}

%% file: classical.tex
\section{Classical groups}
\label{sec:classical-group}
In this section, we will apply Theorem \ref{th:main-result} to
the special linear group, the symplectic group and the orthogonal group.

\subsection[]{Special linear group}
\label{subsec:application-to-sl-n}
Let $G=SL_n$ be the $O$-group subscheme $G\subset GL_n$ defined by 
$$
 G=\{g\in GL_n\mid\det g=1\}.
$$
Its Lie algebra is
$$
 \frak{g}=\frak{sl}_n=\{X\in\frak{gl}_n\mid\text{\rm tr}(X)=0\}.
$$
Then

\begin{prop}\label{prop:sl-n-is-smooth-with-condition-i-ii-iii}
\begin{enumerate}
\item $G=SL_n$ is smooth over $O$,
\item the conditions II) and III) of subsection 
\ref{subsec:fundamental-setting} hold for $G=SL_n$,
\item the condition I) of subsection 
\ref{subsec:fundamental-setting} holds for $G=SL_n$ if and only if 
$n$ is prime to
the characteristic of $\Bbb F$.
\end{enumerate}
\end{prop}
\begin{proof}
1) Take any $O$-algebra $R$ and an ideal $\frak{a}\subset R$ such that
$\frak{a}^2=0$. For any $g\npmod{\frak a}\in G(R/\frak{a})$ 
($g\in M_n(R)$), we have $\det g=1+a$ with $a\in\frak{a}$. Then 
$(1-a)(1+a)=1$ because $a^2\in\frak{a}^2=0$. Now
$$
 h=\begin{bmatrix}
    1-a& &      & \\
       &1&      & \\
       & &\ddots& \\
       & &      &1
   \end{bmatrix}\cdot g\in G(R)
$$
and $h\equiv g\npmod{\frak a}$. Hence the canonical mapping 
$G(R)\to G(R/\frak{a})$ is surjective, which means that 
$G=SL_n$ is smooth
over $O$ (see \cite[p.111, Cor. 4.6]{Demazure-Gabriel1970}).

2) Let $r=l+l^{\prime}$ with $l\geq l^{\prime}>0$. Then 
for any $X\in M_n(O)$ with the eigenvalues 
$\alpha_i$ ($1\leq i\leq n$), we have
\begin{align*}
 \det(1_n+\varpi^lX)
 &\equiv\prod_{i=1}^n(1+\varpi^l\alpha_i)\npmod{\frak{p}^r}\\
 &\equiv 1+\varpi^l\cdot\text{\rm tr}(X)\npmod{\frak{p}^r}
\end{align*}
because $2l\geq r$. Then 
$\det(1_n+\varpi^lX)\equiv 1\npmod{\frak{p}^r}$ if and only if 
$\text{\rm tr}(X)\equiv 0\npmod{\frak{p}^{l^{\prime}}}$. Hence the
  condition II) holds.
If $r=2l-1>1$ is odd, then we have
\begin{align*}
 &\det\left(1_n+\varpi^{l-1}X+2^{-1}\varpi^{2l-2}X^2\right)\\
 &=\prod_{i=1}^n\left(
    1+\varpi^l\alpha_i+2^{-1}\varpi^{2l-2}\alpha_i^2\right)\\
 &\equiv 1+\varpi^{l-1}\sum_{i=1}^n\alpha_i
  +\varpi^{2l-2}\left(\sum_{i<j}\alpha_i\alpha_j
                      +2^{-1}\sum_{i=1}^n\alpha_i^2\right)
  \npmod{\frak{p}^r}\\
 &\equiv 1+\varpi^{l-1}\text{\rm tr}(X)
   +2^{-1}\varpi^{2l-2}\left(\text{\rm tr}\,X\right)^2
  \npmod{\frak{p}^r}
\end{align*}
because $3l-3\geq r$. Then 
$1_n+\varpi^{l-1}X+2^{-1}\varpi^{2l-2}X^2\npmod{\frak{p}^r}
 \in G(O_r)$ if $\text{\rm tr}(X)=0$. Hence the condition III) holds.

3) Take a $X\in\frak{g}(\Bbb F)$ such that $\text{\rm tr}(XY)=0$ for
all $Y\in\frak{g}(\Bbb F)$. Then $X=x\cdot 1_n$ with 
$\text{\rm tr}(X)=0$. Because $n$ is prime to the characteristic of
$\Bbb F$, we have $X=0$.
\end{proof}

Take a $\beta\in\frak{g}(O)=\frak{sl}_n(O)$ such that 
the
characteristic polynomial of $\overline\beta\in M_n(\Bbb F)$ is its
minimal polynomial and is separable. Then the fibers 
$G_{\beta}{\otimes}_OK$ ($K=\Bbb F, F$) are connected and $\beta$ is
smoothly regular with respect to $G$ over $K=F, \Bbb F$ by the remark
in subsubsection \ref{subsubsec:smooth-regularity-for-sl-n}. 
Hence $G_{\beta}$ is commutative and smooth over $O$ by Proposition
\ref{tprop:sufficient-condition-for-smooth-commutativeness-of-g-beta},
and Theorem \ref{th:main-result} is applicable.

Examples of such $\beta\in\frak{g}(O)=\frak{sl}_n(O)$ is given by an 
unramified separable extension $L/F$ of degree $n$. Identify $L$ with a
$F$-subalgebra of $M_n(F)$ by means of the regular representation with
respect to an $O$-basis of $O_L$. Take any $\beta\in L$ such that 
$O_L=O[\beta]$ and $T_{L/F}(\beta)=0$. Then $\beta\in\frak{sl}_n(O)$
and the
reduction modulo $\frak{p}$ of the characteristic polynomial of
$\beta\in M_n(O)$ is irreducible by Proposition
\ref{prop:generator-of-tame-extension-is-smoothly-regular}. 
In this case, we have
$$
 G_{\beta}(O_r)
 =\left\{\varepsilon\nnpmod{\frak{p}_L^{er}}
              \in\left(O_L/\frak{p}_L^{er}\right)^{\times}
                           \mid \varepsilon\in U_{L/F}\right\}
$$
where $e$ is the ramification index of $L/F$ and
$$
 U_{L/F}=G_{\beta}(O)
 =\{\varepsilon\in O_L^{\times}\mid N_{L/F}(\varepsilon)=1\}.
$$
We have also
\begin{align*}
 &G_{\beta}(O_r)\cap K_l(O_r)\\
 &=\left\{1+\varpi^lx\nnpmod{\frak{p}_L^{er}}
           \in\left(O_L/\frak{p}_L^{er}\right)^{\times}
    \mid x\in O_L, 
         T_{L/F}(x)\equiv 0\npmod{\frak{p}^{l^{\prime}}}
          \right\}
\end{align*}
and 
$$
 \psi_{\beta}\left(1+\varpi^lx\nnpmod{\frak{p}_L^{er}}\right)
 =\tau\left(\varpi^{-l^{\prime}}T_{L/F}(x\beta)\right)
$$ 
for $x\in O_L$
 such that $T_{L/F}(x)\equiv 0\npmod{\frak{p}^{l^{\prime}}}$. Then
 Theorem \ref{th:main-result} gives

\begin{thm}\label{th:shintani-gerardin-parametrization-for-sl-n}
Let $G=SL_n$ with $n$ prime to the characteristic of $\Bbb F$. Then 
there exists a bijection 
$\theta\mapsto
 \text{\rm Ind}_{G(O_r,\beta)}^{G(O_r)}\sigma_{\beta,\theta}$ of the
 set 
$$
  \left\{\begin{array}{l}
        \theta:U_{L/F}\to\left(O_L/\frak{p}_L^{er}\right)^{\times}
               \to\Bbb C^{\times}:\text{\rm group homomorphism}\\
   \hphantom{\theta:}
        \text{\rm s.t. 
  $\theta(\gamma)=\tau(\varpi^{-l^{\prime}}T_{L/F}(\beta x))$}\\
   \hphantom{\theta:s.t.}
        \text{\rm for 
   $\forall\gamma\in U_{L/F}$ s.t. 
        $\gamma\equiv 1+\varpi^lx\npmod{\frak{p}_L^{er}},\;x\in O_L$}
         \end{array}\right\}
$$
onto $\text{\rm Irr}(G(O_r)\mid\psi_{\beta})$.
\end{thm}

\subsection[]{Symplectic group}
\label{subsec:tamely-ramified-exension-and-sp}
Let $G=Sp_{2n}$ be the $O$-group scheme such that
$$
 Sp_{2n}(R)=\{g\in GL_{2n}(R)\mid gJ_n\,^t\!g=J_n\}
$$
 ($J_n=\begin{bmatrix}
        0&1_n\\
       -1_n&0
       \end{bmatrix}$) 
for all $O$-algebra $R$. The Lie algebra $\frak{g}=\frak{sp}_{2n}$ of $G$ 
is an affine $O$-subscheme of $\frak{gl}_{2n}$ such that
$$
 \frak{sp}_{2n}(R)=\{X\in\frak{gl}_{2n}(R)\mid XJ_n+J_n\,^t\!X=0\}
$$
for all $O$-algebra $R$. Then 

\begin{prop}\label{prop:sp-2n-is-smooth-with-condition-i-ii-iii}
\begin{enumerate}
\item $G=Sp_{2n}$ is smooth over $O$,
\item the conditions I), II) and III) of subsection 
\ref{subsec:fundamental-setting} hold for $G=Sp_{2n}$.
\end{enumerate}
\end{prop}
\begin{proof}
1) Take any $O$-algebra $R$ and an ideal $\frak{a}\subset R$ such that
$\frak{a}^2=0$. For any $g\npmod{\frak{a}}\in G(R/\frak{a})$ 
($g\in M_{2n}(R)$), we have 
$gJ_n\,^tg=(1_{2n}+X)J_n$ with $X\in M_{2n}(\frak{a})$. Then 
$(1_{2n}+2^{-1}X)^2=1_{2n}+X$ because $X^2=0$. Then we have
$$
 gJ_n\,^tg=(1_{2n}+2^{-1}X)J_n\,^t(1_{2n}+2^{-1}X)
$$
because $XJ_n=J_n\,^tX$. We have also 
$\det(1_{2n}+2^{-1}X)\in R^{\times}$ because 
$\det(1_{2n}+2^{-1}X)\equiv 1\npmod{\frak a}$. Now 
$$
 h=(1_{2n}+2^{-1}X)^{-1}g\in Sp_{2n}(R).
$$
and $g\equiv h\npmod{\frak a}$. Hence the canonical mapping 
$G(R)\to G(R/\frak{a})$ is surjective, which means that 
$G=Sp_{2n}$ is smooth over $O$ 
(see \cite[p.111, Cor. 4.6]{Demazure-Gabriel1970}).

2) For any $T\in M_{2n}(\Bbb F)$, we have
$Y=T+J_n\,^tTJ_n\in\frak{g}(\Bbb F)$ and 
$$
 \text{\rm tr}(XY)=2\text{\rm tr}(XT)
$$
for all $X\in\frak{g}(\Bbb F)$. Hence the condition I) holds.

For any $X\in M_{2n}(O)$, we have 
$$
 (1_{2n}+\varpi^lX)J_n\,^t(1_{2n}+\varpi^lX)
 \equiv 1_{2n}+\varpi^l(XJ_n+J_n\,^tX)\npmod{\frak{p}^r}
$$
which implies that the condition II) holds.

Take any $X\in\frak{g}(O)$. Then we have
$$
 (1_{2n}+\varpi^{l-1}X+2^{-1}\varpi^{2l-2}X^2)J_n\,
 ^t(1_{2n}+\varpi^{l-1}X+2^{-1}\varpi^{2l-2}X^2)
 \equiv J_n\npmod{\frak{p}^r}
$$
because $3l-3\geq r$. So the condition III) holds.
\end{proof}

Take a $\beta\in\frak{g}(O)$ such that the characteristic polynomial of
$\overline\beta\in M_{2n}(\Bbb F)$ is its minimal polynomial and 
$\det\overline\beta\neq 0$. Then the fibers 
$G_{\beta}{\otimes}_OK$ ($K=\Bbb F, F$) are connected and $\beta$ is
smoothly regular with respect to $G$ over $K=F, \Bbb F$ by the remark
in subsubsection \ref{subsubsec:smooth-regularity-for-sp-2n}. 
Hence $G_{\beta}$ is commutative and smooth over $O$ by Proposition
\ref{tprop:sufficient-condition-for-smooth-commutativeness-of-g-beta},
and Theorem \ref{th:main-result} is applicable.

Proposition \ref{prop:generator-of-tame-extension-is-smoothly-regular}
gives examples of such $\beta\in\frak{g}(O)$. 

Let $L_+/F$ be a tamely ramified extension of degree $n$ and $L/L_+$ a
quadratic extension. Take a $\omega\in O_L$ such that 
$$
 O_L=O_{L_+}\oplus\omega O_{L_+},
 \qquad
 \omega^{\rho}=-\omega
$$
where $\rho\in\text{\rm Gal}(L/L_+)$ is the non-trivial element. Then
$$
 D(x,y)
 =\frac 12T_{L/F}\left(\omega^{-1}\varpi_{L_+}^{1-e_+}x^{\rho}y\right)
 \qquad
 (x,y\in L)
$$
with the ramification index $e_+$ of $L_+/F$ and a prime element
$\varpi_{L_+}$ of $L_+$ 
is a symplectic form on the $F$-vector space $L$. Fix an $O$-basis 
$\{u_1,\cdots,u_n\}$ of $O_{L_+}$. Since $L_+/F$ is a tamely ramified
extension, there exists $u_j^{\ast}\in\frak{p}_{L_+}^{1-e_+}$ 
($1\leq j\leq n$) such that $T_{L_+/F}(u_iu_j^{\ast})=\delta_{ij}$. If
we put $v_j=\omega\cdot\varpi_{L_+}^{e_+-1}\cdot u_j^{\ast}\in O_L$,
then we have
$$
 D(u_i,u_j)=D(v_i,v_j)=0,
 \quad
 D(u_i,v_j)=\delta_{ij}
 \quad
 (1\leq i,j\neq n).
$$
Identify the $F$-algebra $L$ with a $F$-subalgebra of $M_{2n}(F)$ by
means of the $O$-basis $\{u_1,\cdots,u_n,v_1,\cdots,v_n\}$ of $O_L$. 

Take a $\beta\in O_L$ such that $O_L=O[\beta]$ and
$\beta^{\rho}+\beta=0$. Then 
\begin{enumerate}
\item $\beta\in\frak{g}(O)$ and the characteristic polynomial of 
      $\overline\beta\in M_{2n}(\Bbb F)$ is its minimal polynomial,
  and
\item $\det\overline\beta\neq 0$ for 
      $\overline\beta\in M_{2n}(\Bbb F)$ 
      if and only if $L/F$ is not totally ramified
\end{enumerate}
by Proposition
\ref{prop:generator-of-tame-extension-is-smoothly-regular}. 
We have
$$
 G_{\beta}(O_r)
 =\left\{\varepsilon\nnpmod{\frak{p}_L^{er}}
           \in\left(O_L/\frak{p}_L^{er}\right)^{\times}
   \mid\varepsilon\in U_{L/L_+}\right\}
$$
where $e$ is the ramification index of $L/F$ and
$$
 U_{L/L_+}=G_{\beta}(O)
 =\{\varepsilon\in O_L^{\times}\mid N_{L/L_+}(\varepsilon)=1\}.
$$
We have also
\begin{align*}
 &G_{\beta}(O_r)\cap K_l(O_r)\\
 &=\left\{1+\varpi^lx\nnpmod{\frak{p}_L^{er}}
                \in\left(O_L/\frak{p}_L^{er}\right)^{\times}
    \mid x\in O_L, 
         T_{L/L_+}(x)\equiv 0\npmod{\frak{p}_{L_+}^{e_+l^{\prime}}}
          \right\}
\end{align*}
and 
$$
 \psi_{\beta}\left(1+\varpi^lx\nnpmod{\frak{p}_L^{er}}\right)
 =\tau\left(\varpi^{-l^{\prime}}T_{L/F}(\beta x)\right)
$$
for $x\in O_L$ such that 
$T_{L/L_+}(x)\equiv 0\npmod{\frak{p}_{L_+}^{e_+l^{\prime}}}$. 
Then Theorem \ref{th:main-result} gives

\begin{thm}\label{thm:shintani-gerardin-type-parametrization-for-sp-2n}
Assume that $L/F$ is not totally ramified. Then 
there exists a bijection 
$\theta\mapsto
 \text{\rm Ind}_{G(O_r,\beta)}^{G(O_r)}\sigma_{\beta,\theta}$ of the
 set 
$$
 \left\{\begin{array}{l}
        \theta:U_{L/L_+}\to\left(O_L/\frak{p}_L^{er}\right)^{\times}
               \to\Bbb C^{\times}:\text{\rm group homomorphism}\\
   \hphantom{\theta:}
        \text{\rm s.t. 
  $\theta(\gamma)=\tau(\varpi^{-l^{\prime}}T_{L/F}(\beta x))$}\\
   \hphantom{\theta:s.t.}
        \text{\rm for 
   $\forall\gamma\in U_{L/L_+},
          \gamma\equiv 1+\varpi^lx\npmod{\frak{p}_L^{er}},\;x\in O_L$}
         \end{array}\right\}
$$
onto $\text{\rm Irr}(G(O_r)\mid\psi_{\beta})$.
\end{thm}

\subsection[]{Orthogonal group}
\label{subsec:tamely-ramified-exension-and-so}
Take a $S\in M_m(O)$ such that $^tS=S$ and $\det S\in O^{\times}$. 
Let $G=SO(S)$ be the $O$-group subscheme of $SL_m$ such that 
$$
 G(R)=\{g\in SL_m(R)\mid gS^t\!g=S\}
$$
for all $O$-algebra $R$. The Lie algebra $\frak{g}=\frak{so}(S)$ of
$G$ is an affine $O$-subscheme of $\frak{gl}_m$ such that
$$
 \frak{g}(R)=\{X\in\frak{gl}_m(R)\mid XS+S\,^t\!X=0\}
$$
for all $O$-algebra $R$. Then 

\begin{prop}\label{prop:so(s)-is-smooth-with-condition-i-ii-iii}
\begin{enumerate}
\item $G=SO(S)$ is smooth over $O$,
\item the conditions I), II) and III) of subsection 
      \ref{subsec:fundamental-setting} hold for $G=SO(S)$.
\end{enumerate}
\end{prop}
\begin{proof}
1) Take any $O$-algebra $R$ and an ideal $\frak{a}\subset R$ such that
$\frak{a}^2=0$. For any $g\npmod{\frak a}\in G(R/\frak{a})$, we have 
$gS\,^tg=(1_m+X)S$ with $X\in M_m(\frak{a})$. Then, as in the proof of
Proposition \ref{prop:sp-2n-is-smooth-with-condition-i-ii-iii}, we
have $\det(1_m+2^{-1}X)\in R^{\times}$ and put 
$h=(1_m+2^{-1}X)^{-1}g\in GL_m(R)$. Then we have 
$hS\,^th=S$ and $h\equiv g\npmod{\frak a}$. If $\det h=-1$, then
$2\in\frak{a}$ because $\det h\equiv\det g\equiv 1\npmod{\frak a}$,
this means $\frak{a}=R$ since $F$ is non-dyadic. Hence the canonical
mapping $G(R)\to G(R/\frak{a})$ is surjective, which means that 
$G=SO(S)$ is smooth over $O$ 
(see \cite[p.111, Cor. 4.6]{Demazure-Gabriel1970}).

2) For any $T\in M_{2n}(\Bbb F)$, we have
$Y=T-S\,^tTS^{-1}\in\frak{g}(\Bbb F)$ and 
$$
 \text{\rm tr}(XY)=2\text{\rm tr}(XT)
$$
for all $X\in\frak{g}(\Bbb F)$. Hence the condition I) holds. 
Similar arguments as in the proof of Proposition 
\ref{prop:sp-2n-is-smooth-with-condition-i-ii-iii} show that the
conditions II) and III) hold for $G=SO(S)$.
\end{proof}

\subsubsection[]{Case of even variables}
\label{subsubsec:tamely-ramified-extension-and-so-2n}
Let us consider the case of $m=2n$ being even. 

Take a $\beta\in\frak{g}(O)$ such that the characteristic polynomial
of $\overline\beta\in M_m(\Bbb F)$ is its minimal polynomial and 
$\det\overline\beta\neq 0$. Then 
$\beta\in\frak{g}(O)$ is smoothly regular with respect to $G=SO(S)$
over $O$ by the remark in subsubsection 
\ref{subsubsec:smooth-regularity-for-so-2n}, and the fibers 
$G_{\beta}{\otimes}_OK$ ($K=\Bbb F, F$) are connected. So 
$G_{\beta}$ is smooth over $O$ and Theorem \ref{th:main-result} is
applicable. 

Proposition \ref{prop:generator-of-tame-extension-is-smoothly-regular}
gives examples of such $\beta\in\frak{g}(O)$. 

Let $L/F$ be a tamely ramified Galois
extension of degree $2n$. Fix an intermediate field 
$F\subset L_+\subset L$ such that $(L:L_+)=2$, 
and assume that $L/L_+$ is unramified. Take an  
$\varepsilon\in O_{L_+}^{\times}$ and put
$$
 S_{\varepsilon}(x,y)
 =T_{L/F}\left(\varepsilon\cdot\varpi_{L_+}^{1-e}\cdot xy^{\rho}\right)
 \qquad
 (x,y\in L)
$$
where $\rho\in\text{\rm Gal}(L/L_+)$ is the non-trivial element, 
$e$ is the ramification index of $L/F$ and $\varpi_{L_+}$ is a prime
element of $L_+$. 
Then $S_{\varepsilon}$ is a regular $F$-quadratic form on $L$. Take a
$O$-basis $\{u_1,\cdots,u_{2n}\}$ of $O_L$ and put 
$B=\left(u_i^{\sigma_j}\right)_{1\leq i,j\leq 2n}$ with
$\text{\rm Gal}(L/F)=\{\sigma_1,\cdots,\sigma_{2n}\}$. Then we have
$$
 \left(S_{\varepsilon}(u_i,u_j)\right)_{1\leq i,j\leq 2n}
 =B\begin{bmatrix}
    (\varepsilon\varpi_{L_+}^{1-e})^{\sigma_1}&  &  \\
                                    &\ddots&  \\
       &  &(\varepsilon\varpi_{L_+}^{1-e})^{\sigma_{2n}}
   \end{bmatrix}
  \,^tB^{\rho}
$$
so that the discriminant of the quadratic form $S_{\varepsilon}$ is
$$
 \det\left(S_{\varepsilon}(u_i,u_j)\right)_{1\leq i,j\leq 2n}
 =\pm(\det B)^2N_{L/F}\left(\varepsilon\varpi_{L_+}^{1-e}\right).
$$
Note that $(\det B)^{\sigma}=\pm\det B$ for any 
$\sigma\in\text{\rm Gal}(L/F)$. Since $L/F$ is tamely ramified, its
discriminant is 
$$
 D(L/F)=(\det B^2)=\frak{p}^{f(e-1)}
$$
where $2n=ef$. Hence 
$\det\left(S_{\varepsilon}(u_i,u_j)\right)_{1\leq i,j\leq 2n}
 \in O^{\times}$. So
the $O$-group scheme $G=SO(S_{\varepsilon})$ and its Lie algebra
$\frak{g}=\frak{so}(S_{\varepsilon})$ is 
defined by 
$$
 G(R)=\left\{g\in SL_R(O_L{\otimes}R)\biggm|
                    \begin{array}{l}
                    S_{\varepsilon}(xg,yg)=S_{\varepsilon}(x,y)\\
                    \text{\rm for $\forall x,y\in O_L{\otimes}_OR$}
                    \end{array}\right\}
$$
and by 
$$
 \frak{g}(R)
  =\left\{X\in\text{\rm End}_R(O_L{\otimes}R)\biggm|
           \begin{array}{l}
            S_{\varepsilon}(xX,y)+S_{\varepsilon}(x,yX)=0\\
            \text{\rm for $\forall x,y\in O_L{\otimes}_OR$}
           \end{array}\right\}
$$
for all $O$-algebra $R$. Note that $\text{\rm End}_F(L)$
acts on $L$ from the right side.

Take a $\beta\in O_L$ such that $O_L=O[\beta]$ and 
$\beta^{\rho}+\beta=0$. Since $L/F$ is not totally ramified, 
Proposition \ref{prop:generator-of-tame-extension-is-smoothly-regular}
implies that $\beta$ is an unit of $O_L$. 
Identify $\beta\in L$ with the element 
$x\mapsto x\beta$ of $\frak{g}(O)\subset\text{\rm End}_O(O_L)$. Then
we have
$$
 G_{\beta}(O_r)
 =\left\{\varepsilon\nnpmod{\frak{p}_L^{er}}
           \in\left(O_L/\frak{p}_L^{er}\right)^{\times}
    \mid\varepsilon\in U_{L/L_+}\right\}
$$
where $e$ is the ramification index of $L/F$ and
$$
 U_{L/L_+}=G_{\beta}(O)
 =\{\varepsilon\in O_L^{\times}\mid N_{L/L_+}(\varepsilon)=1\}.
$$
We have also
\begin{align*}
 &G_{\beta}(O_r)\cap K_l(O_r)\\
 &=\left\{1+\varpi^lx\nnpmod{\frak{p}_L^{er}}
           \in\left(O_L/\frak{p}_L^{er}\right)^{\times}
    \mid x\in O_L, 
     T_{L/L_+}(x)\equiv 0\npmod{\frak{p}_{L_+}^{el^{\prime}}}
         \right\}
\end{align*}
and 
$$
 \psi_{\beta}\left(1+\varpi^lx\nnpmod{\frak{p}_L^{er}}\right)
 =\tau\left(\varpi^{-l^{\prime}}T_{L/F}(\beta x)\right)
$$
for $x\in O_L$ such that 
$T_{L/L_+}(x)\equiv 0\npmod{\frak{p}_{L_+}^{e_+l^{\prime}}}$. 
Then Theorem \ref{th:main-result} gives

\begin{thm}\label{thm:shintani-gerardin-type-parametrization-for-even-so}
There exists a bijection 
$\theta\mapsto
 \text{\rm Ind}_{G(O_r,\beta)}^{G(O_r)}\sigma_{\beta,\theta}$ of the
 set 
$$
 \left\{\begin{array}{l}
        \theta:U_{L/L_+}\to\left(O_L/\frak{p}_L^{er}\right)^{\times}
               \to\Bbb C^{\times}:\text{\rm group homomorphism}\\
   \hphantom{\theta:}
        \text{\rm s.t. 
  $\theta(\gamma)=\tau(\varpi^{-l^{\prime}}T_{L/F}(\beta x))$}\\
   \hphantom{\theta:s.t.}
        \text{\rm for
    $\forall\gamma\in U_{L/L_+},
          \gamma\equiv 1+\varpi^lx\npmod{\frak{p}_L^{er}},\;x\in O_L$}
         \end{array}\right\}
$$
onto $\text{\rm Irr}(G(O_r)\mid\psi_{\beta})$.
\end{thm}

\subsubsection[]{Case of odd variables}
\label{subsubsec:tamely-ramified-extension-and-so-2n+1}
Let us consider the case of $m=2n+1$ being odd. 

Take a $\beta\in\frak{g}(O)$ such that the characteristic polynomial
of $\overline\beta\in M_m(\Bbb F)$ is its minimal polynomial. Then 
$\beta\in\frak{g}(O)$ is smoothly regular with respect to $G=SO(S)$
over $O$ by the remark in subsubsection 
\ref{subsubsec:smooth-regularity-for-so-an+1}, and the fibers 
$G_{\beta}{\otimes}_OK$ ($K=\Bbb F, F$) are connected. So 
$G_{\beta}$ is smooth over $O$ and Theorem \ref{th:main-result} is
applicable. 

Proposition \ref{prop:generator-of-tame-extension-is-smoothly-regular}
gives examples of such $\beta\in\frak{g}(O)$. 

Use the notations of
the preceding subsubsection. Take a 
$\eta\in O^{\times}$ and define a $F$-quadratic form
$S_{\varepsilon,\eta}$ on the $F$-vector space $L\times F$ by
$$
 S_{\varepsilon,\eta}((x,s),(y,t))
 =S_{\varepsilon}(x,y)+\eta\cdot st.
$$
Then the $O$-group scheme $G=SO(S_{\varepsilon,\eta})$ and its Lie
algebra $\frak{g}=\frak{so}(S_{\varepsilon,\eta})$ is defined by
$$
 G(R)=\left\{g\in SL_R((O_L\times O){\otimes}R)\biggm|
                    \begin{array}{l}
                    S_{\varepsilon,\eta}(ug,vg)
                         =S_{\varepsilon,\eta}(u,v)
                    \;\;\text{\rm for}\\
                    \text{\rm $\forall u,v\in(O_L\times O){\otimes}_OR$}
                    \end{array}\right\}
$$
and
$$
 \frak{g}(R)
  =\left\{X\in\text{\rm End}_R((O_L\times O){\otimes}R)\biggm|
           \begin{array}{l}
            S_{\varepsilon,\eta}(uX,v)+S_{\varepsilon,\eta}(u,vX)=0\\
            \text{\rm for $\forall u,v\in(O_L\times O){\otimes}_OR$}
           \end{array}\right\}
$$
for all $O$-algebra $R$. An element 
$X\in\text{\rm End}_R((O_L\times O){\otimes}_OR)$ is denoted by
$$
 X=\begin{bmatrix}
    A&B\\
    C&D
   \end{bmatrix}\;\text{\rm with}\;
 \begin{cases}
  A\in\text{\rm End}_R(O_L{\otimes}_OR),
    &B\in\text{\rm Hom}_R(O_L{\otimes}_OR,R),\\
  C\in\text{\rm Hom}_R(R,O_L{\otimes}_OR),
    &D\in\text{\rm End}_R(R)=R.
 \end{cases}
$$
Take an $\alpha\in O_L$ such that $O_L=O[\alpha]$ and 
$\alpha^{\rho}+\alpha=0$. Identify $\alpha\in O_L$
with the endomorphism 
$[x\mapsto x\alpha]\in \text{\rm End}_O(O_L)$. Then
$$
 \beta=\begin{bmatrix}
        \alpha&0\\
          0   &0
       \end{bmatrix}\in\frak{g}(O)
$$
and the characteristic polynomial of 
$\overline\beta=\beta\npmod{\frak{p}}
 \in\text{\rm End}_{\Bbb F}(\Bbb L\times\Bbb F)$ 
is its minimal polynomial by Proposition 
\ref{prop:generator-of-tame-extension-is-smoothly-regular}. 
In this case we have
$$
 G_{\beta}(O_r)
  =\left\{\begin{bmatrix}
           \gamma\npmod{\frak{p}_L^{er}}&0\\
           0&1
          \end{bmatrix}\biggm|\gamma\in U_{L/L_+}\right\}
$$
where
$$
 U_{L/L_+}=\{\varepsilon\in O_L^{\times}\mid
               N_{L/L_+}(\varepsilon)=1\}.
$$
We have also 
$\psi_{\beta}(h)=\tau\left(\varpi^{-l^{\prime}}T_{L/F}(\alpha x)\right)$
for all
$$
 h=\begin{bmatrix}
    1+\varpi^lx\npmod{\frak{p}_L^r}&0\\
    0&1
   \end{bmatrix}
 \in K_l(O_r)\cap G_{\beta}(O_r).
$$
Then Theorem \ref{th:main-result} gives 

\begin{thm}\label{thm:shintani-gerardin-type-parametrization-for-odd-so}
There exists a bijection 
$\theta\mapsto
 \text{\rm Ind}_{G(O_r,\beta)}^{G(O_r)}\sigma_{\beta,\theta}$ of the
 set 
$$
 \left\{\begin{array}{l}
        \theta:U_{L/L_+}\to\left(O_L/\frak{p}_L^{er}\right)^{\times}
               \to\Bbb C^{\times}:\text{\rm group homomorphism}\\
   \hphantom{\theta:}
        \text{\rm s.t. 
  $\theta(\gamma)=\tau(\varpi^{-l^{\prime}}T_{L/F}(\alpha x))$}\\
   \hphantom{\theta:s.t.}
        \text{\rm for 
   $\forall\gamma\in U_{L/L_+},
          \gamma\equiv 1+\varpi^lx\npmod{\frak{p}_L^{er}},\;x\in O_L$}
         \end{array}\right\}
$$
onto $\text{\rm Irr}(G(O_r)\mid\psi_{\beta})$.
\end{thm}